\documentclass[a4paper,reqno]{amsart}

\usepackage{fullpage}

\usepackage[T1]{fontenc}
\usepackage{lmodern}
\usepackage[utf8]{inputenc}

\usepackage{amssymb}
\usepackage[dvipsnames]{xcolor}

\usepackage[colorlinks]{hyperref}
\hypersetup{
  allcolors = [rgb]{0,0,0.5}
}
\usepackage[hyperpageref]{backref}
\usepackage[nobysame,alphabetic,initials]{amsrefs}

\usepackage{tikz}
\usetikzlibrary{cd}

\DefineSimpleKey{bib}{how}
\DefineSimpleKey{bib}{mrclass}
\DefineSimpleKey{bib}{mrnumber}
\DefineSimpleKey{bib}{fjournal}
\DefineSimpleKey{bib}{mrreviewer}

\renewcommand{\PrintDOI}[1]{%
  \href{http://dx.doi.org/#1}{{\tt DOI:#1}}%
}
\renewcommand{\eprint}[1]{#1}

\BibSpec{book}{%
    +{}  {\PrintPrimary}                {transition}
    +{.} { \PrintDate}                  {date}
    +{.} { \textit}                     {title}
    +{.} { }                            {part}
    +{:} { \textit}                     {subtitle}
    +{,} { \PrintEdition}               {edition}
    +{}  { \PrintEditorsB}              {editor}
    +{,} { \PrintTranslatorsC}          {translator}
    +{,} { \PrintContributions}         {contribution}
    +{,} { }                            {series}
    +{,} { \voltext}                    {volume}
    +{,} { }                            {publisher}
    +{,} { }                            {organization}
    +{,} { }                            {address}
    +{,} { }                            {status}
    +{,} { \PrintDOI}                   {doi}
    +{,} { \PrintISBNs}                 {isbn}
    +{}  { \parenthesize}               {language}
    +{}  { \PrintTranslation}           {translation}
    +{;} { \PrintReprint}               {reprint}
    +{.} { }                            {note}
    +{.} {}                             {transition}
}
\BibSpec{article}{%
    +{}  {\PrintAuthors}                {author}
    +{,} { \textit}                     {title}
    +{.} { }                            {part}
    +{:} { \textit}                     {subtitle}
    +{,} { \PrintContributions}         {contribution}
    +{.} { \PrintPartials}              {partial}
    +{,} { }                            {journal}
    +{}  { \textbf}                     {volume}
    +{}  { \PrintDatePV}                {date}
    +{,} { \issuetext}                  {number}
    +{,} { \eprintpages}                {pages}
    +{,} { }                            {status}
    +{,} { \PrintDOI}                   {doi}
    +{,} { \eprint}        {eprint}
    +{}  { \parenthesize}               {language}
    +{}  { \PrintTranslation}           {translation}
    +{;} { \PrintReprint}               {reprint}
    +{.} { }                            {note}
    +{.} {}                             {transition}
}
\BibSpec{collection.article}{%
    +{}  {\PrintAuthors}                {author}
    +{,} { \textit}                     {title}
    +{.} { }                            {part}
    +{:} { \textit}                     {subtitle}
    +{,} { \PrintContributions}         {contribution}
    +{,} { \PrintConference}            {conference}
    +{}  {\PrintBook}                   {book}
    +{,} { }                            {booktitle}
    +{,} { \PrintDateB}                 {date}
    +{,} { pp.~}                        {pages}
    +{,} { }                            {publisher}
    +{,} { }                            {organization}
    +{,} { }                            {address}
    +{,} { }                            {status}
    +{,} { \PrintDOI}                   {doi}
    +{,} { \eprint}        {eprint}
    +{}  { \parenthesize}               {language}
    +{}  { \PrintTranslation}           {translation}
    +{;} { \PrintReprint}               {reprint}
    +{.} { }                            {note}
    +{.} {}                             {transition}
}
\BibSpec{misc}{%
  +{}{\PrintAuthors}  {author}
  +{,}{ \textit}      {title}
  +{.}{ }             {how}
  +{}{ \parenthesize} {date}
  +{,} { available at \eprint}        {eprint}
  +{,}{ available at \url}{url}
  +{,}{ }             {note}
  +{.}{}              {transition}
}

\usepackage{amsthm}
\theoremstyle{plain}
\newtheorem{theorem}{Theorem}[section]
\newtheorem{proposition}[theorem]{Proposition}
\newtheorem{lemma}[theorem]{Lemma}
\newtheorem{corollary}[theorem]{Corollary}

\newtheorem{theoremA}{Theorem}

\theoremstyle{definition} 
\newtheorem{definition}[theorem]{Definition}

\theoremstyle{remark}
\newtheorem{example}[theorem]{Example}
\newtheorem{remark}[theorem]{Remark}

\numberwithin{equation}{section}

\mathchardef\mhyph="2D

\newcommand{\rE}{\mathrm{E}}

\newcommand{\rH}{\mathrm{H}}

\newcommand{\rK}{\mathrm{K}}

\newcommand{\cB}{\mathcal{B}}
\newcommand{\cC}{\mathcal{C}}
\newcommand{\cD}{\mathcal{D}}

\newcommand{\cF}{\mathcal{F}}

\newcommand{\cI}{\mathcal{I}}

\newcommand{\cK}{\mathcal{K}}

\newcommand{\cN}{\mathcal{N}}
\newcommand{\cO}{\mathcal{O}}
\newcommand{\cP}{\mathcal{P}}

\newcommand{\cR}{\mathcal{R}}

\newcommand{\cW}{\mathcal{W}}

\newcommand{\dC}{\mathbb{C}}

\newcommand{\dH}{\mathbb{H}}

\newcommand{\dN}{\mathbb{N}}

\newcommand{\dQ}{\mathbb{Q}}
\newcommand{\dR}{\mathbb{R}}
\newcommand{\dS}{\mathbb{S}}
\newcommand{\dT}{\mathbb{T}}

\newcommand{\dZ}{\mathbb{Z}}

\newcommand{\bA}{\mathbf{A}}

\newcommand{\bE}{\mathbf{E}}
\newcommand{\bF}{\mathbf{F}}

\newcommand{\bH}{\mathbf{H}}

\newcommand{\bK}{\mathbf{K}}

\newcommand{\bM}{\mathbf{M}}

\newcommand{\Top}{\mathsf{Top}}
\newcommand{\sSet}{\mathsf{sSet}}
\newcommand{\Ab}{\mathrm{Ab}}
\newcommand{\Spec}{\mathsf{Spec}}
\newcommand{\TopSpec}{\mathsf{TopSpec}}
\newcommand{\TopCGSpec}{\mathsf{Top}_\mathrm{cg}\mathsf{Spec}}
\newcommand{\Emb}{\mathrm{Emb}}
\newcommand{\Span}{\mathsf{Span}}
\newcommand{\Sets}{\mathrm{Sets}}
\newcommand{\FinSets}{\mathrm{FinSets}}
\newcommand{\Path}{\mathsf{Path}}

\newcommand{\CstAlg}{\mathrm{C}^*\mhyph\mathrm{alg}}

\newcommand{\GrCorr}{\mathsf{GrCorr}}
\newcommand{\PrCorr}{\mathsf{Corr}_{\mathrm{pr}}}
\newcommand{\PrGrCorr}{\mathsf{GrCorr}_{\mathrm{pr}}}

\newcommand{\id}{\mathrm{id}}
\newcommand{\ho}{\mathrm{ho}}
\DeclareMathOperator*{\hocolim}{hocolim}

\DeclareMathOperator{\rank}{rank}

\newcommand{\KK}{\mathrm{KK}}
\newcommand{\kk}{\mathrm{kk}}
\newcommand{\icKK}{\mathsf{KK}}
\newcommand{\ickk}{\mathsf{kk}}

\newcommand{\ch}{\mathrm{ch}}
\newcommand{\hc}{\mathrm{hc}}

\newcommand{\norm}[1]{\left\| #1\right\|}
\newcommand{\absv}[1]{\left| #1\right|}
\newcommand{\gr}[1]{\lvert #1\rvert}
\newcommand{\hatotimes}{\mathbin{\hat{\otimes}}}
\newcommand{\opo}{\mathrm{op}}
\newcommand{\KU}{\mathbf{KU}}
\newcommand{\Ex}{\mathrm{Ex}}

\newcommand{\Cl}{\mathrm{Cl}}
\newcommand{\pret}{\mathrm{pr,\acute{e}t}}
\newcommand{\TDLC}{\mathrm{TDLC}}
\newcommand{\TDC}{\mathrm{TDC}}
\newcommand{\LCH}{\mathrm{LCH}}
\newcommand{\ND}{N^{\mathrm{D}}}
\newcommand{\BGL}{\mathrm{BGL}}

\DeclareMathOperator{\Tor}{Tor}
\DeclareMathOperator{\ran}{ran}
\DeclareMathOperator{\Fun}{Fun}
\DeclareMathOperator{\Ind}{Ind}
\DeclareMathOperator{\Res}{Res}
\DeclareMathOperator{\Mor}{Mor}
\DeclareMathOperator{\Map}{Map}
\DeclareMathOperator{\map}{\mathbf{map}}
\DeclareMathOperator{\Sing}{Sing}
\DeclareMathOperator{\Hom}{Hom}

\DeclareMathOperator*{\colim}{colim}

\DeclareMathOperator{\uCom}{uCom}

\author{Valerio Proietti}
\address{Department of Mathematics, University of Oslo, P.O. box 1053, Blindern, 0316 Oslo, Norway}
\email{valeriop@math.uio.no}

\author{Makoto Yamashita}
\address{Department of Mathematics, University of Oslo, P.O. box 1053, Blindern, 0316 Oslo, Norway}
\email{makotoy@math.uio.no}

\title{Chern character for torsion-free ample groupoids}

\begin{document}

\begin{abstract}
For an ample groupoid with torsion-free stabilizers, we construct a Chern character map going from the domain of the Baum--Connes assembly map of $G$ to the Crainic--Moerdijk homology groups of $G$ with rational coefficients.
Assuming the (rational) Baum--Connes conjecture, this implies that the operator $\rK$-groups of the groupoid C$^*$-algebra are rationally isomorphic to the periodicized groupoid homology groups, confirming a variation of Matui's HK conjecture.
Our construction hinges on the recent $\infty$-categorical viewpoint on bivariant $\rK$-theory.
We also present applications to the homology of hyperbolic dynamical systems, the homology of topological full groups, the homotopy type of the algebraic $\rK$-theory spectrum of ample groupoids, and the Elliott invariant of classifiable C$^*$-algebras.
\end{abstract}

\makeatletter
\@namedef{subjclassname@2020}{\textup{2020} Mathematics Subject Classification}
\makeatother

\date{v2: April 9, 2026; v1: September 9, 2025}
\subjclass[2020]{46L85; 19K35, 46M20, 37B02}
\keywords{groupoid, C$^*$-algebra, $K$-theory, homology, Baum--Connes conjecture, HK conjecture, Chern character, topological dynamics.}

\maketitle
\setcounter{tocdepth}{1}
\tableofcontents

\section*{Introduction}

The Chern--Connes character and the Baum--Connes conjecture are cornerstones of noncommutative geometry that connect operator $\rK$-groups of C$^*$-algebras, objects rooted in functional analysis, with geometric and algebro-topological concepts such as homological invariants of classifying spaces~\cites{BCH:class,connes:ncg}.
While these are best understood in the context of groups and manifolds separately via connections to more classical invariants, various generalizations have been proposed as noncommutative geometry matured as a field.
Notably, the framework of topological groupoids is a natural generalization of both spaces and groups, and encompasses more general dynamical systems such as (for example) foliated manifolds and systems arising in symbolic dynamics.

On one hand, although the Baum--Connes conjecture is known to fail in general for groupoids~\cite{MR1911663}, it is still confirmed for a broad class of groupoids by Tu's result~\cite{tu:moy} based on a breakthrough by Higson and Kasparov~\cite{higkas:bc}.
On the other hand, when we can reduce the system to an étale groupoid, the homology theory developed by Crainic and Moerdijk~\cite{cramo:hom} provides a promising homological algebraic counterpart.
Indeed, in the case of smooth groupoids, the Chern--Connes character map to the cyclic homology of smooth convolution algebra factors through this homology~\cite{MR1706117}.

Besides the interest from noncommutative geometry, groupoids turned out to be useful in the structure theory of C$^*$-algebras too.
Thanks to the work of Kumjian and Renault~\cites{kum:diag,ren:cartan}, any separable C$^*$-algebra $A$ with a Cartan subalgebra admits a groupoid model, i.e., $A$ is isomorphic to a (possibly twisted) groupoid C$^*$-algebra.
By a recent result of Li~\cite{li:diag}, this class contains the algebras targeted by Elliott's classification program.
Moreover, the Baum--Connes conjecture holds for these groupoids thanks to amenability.
From this viewpoint, it is desirable to have a systematic way to study $\rK$-groups of such C$^*$-algebras through groupoid homology.

\smallskip

In this paper we look at such problems for torsion-free ample groupoids, namely étale groupoids with totally disconnected unit space $G^{(0)}$, and torsion-free isotropy groups at each point of $G^{(0)}$.
This class of groupoids is prevalent in the theory of dynamical systems, including symbolic dynamics, generalized solenoids, self-similar group actions, tiling spaces, and a class of hyperbolic dynamical systems (Smale spaces)~\cites{kelput:til,nekra:crelle,bm:hom}. 

Our main contribution is the construction of a generalized Chern character map, going from the domain of the Baum--Connes assembly map for $G$, which we denote below by $\rK_*^{\mathrm{top}}(G; C_0(G^{(0)}))$, to the homology groups of $G$ with coefficients in the constant sheaf of rational numbers, denoted $\rH_{\ast}(G; \dQ)$.

\begin{theoremA}[Theorem~\ref{thm:ch-ab}]\label{thm:a-def-ch}
Let $G$ be a second-countable, locally compact, Hausdorff, and ample groupoid. Assume that the stabilizers $G^x_x$ for $x \in G^{(0)}$ are torsion-free.
There is a homomorphism of abelian groups
\begin{equation*}
\rK_i^{\mathrm{top}}(G; C_0(G^{(0)})) \to \bigoplus_{k\in \dN} \rH_{i + 2k}(G; \dQ) \quad (i = 0, 1),
\end{equation*}
which becomes an isomorphism after rationalization.
\end{theoremA}

The leading application is a description of the rationalized $\rK$-groups for the groupoid C$^*$-algebra under the assumption that the rational Baum--Connes conjecture holds.

\begin{theoremA}[Corollary~\ref{cor:hk-iso}]\label{thm:main-thm}
In addition to the assumptions of Theorem~\ref{thm:a-def-ch}, assume that $G$ satisfies the rational Baum--Connes conjecture.
Then the Chern character of $G$ induces an isomorphism
\begin{equation*}
\rK_i(C^*_r G)\otimes \dQ \cong \bigoplus_{k\in \dN} \rH_{i+2k}(G; \dQ).
\end{equation*}
\end{theoremA}

This proves a variation of Matui's \emph{HK conjecture}~\cite{matui:hkpub}*{Conjecture 2.6}, whose original version did not involve rationalization and had different assumptions on the groupoid $G$.
This conjecture has attracted a great deal of interest among the C$^*$-algebra community in recent years.

In the positive direction, Matui himself proved it for groupoids arising from shifts of finite type, AF groupoids, and $\dZ$-actions on the Cantor set~\cite{matui:hk}.
This was further carried out for certain Katsura--Exel--Pardo groupoids~\cite{ort:kep}, one-dimensional solenoids~\cite{yi:sol}, and the Renault--Deaconu groupoids of low-rank~\cite{fkps:hk}.
For general ample groupoids with torsion-free stabilizers, we obtained a convergent spectral sequence
\begin{equation}\label{eq:ss}
E^2_{pq} = \rH_p(G; \rK_q(\dC))\Rightarrow \rK_{p + q}^{\mathrm{top}}(G; C_0(G^{(0)}))
\end{equation}
based on the triangulated structure of equivariant $\KK$-theory~\cite{py:part-one}.
This proves the HK isomorphism for all groupoids as in Theorem~\ref{thm:main-thm}, with an additional assumption of low homological dimension.
Moreover, a concrete criterion for the vanishing of higher homology was given in~\cite{bdgw:matui}.

Despite these positive results, the original formulation of the HK conjecture fails in general due to counterexamples by Scarparo and Deeley~\cites{scarparo:homod,deeley:counthk}.
Soon afterwards the suggested fix by rationalization emerged as a ``folk conjecture'', see~\cite{bdgw:matui}*{Remark 5.12}.
Our result shows that, for the class of groupoids considered, this conjecture is indeed true and~\eqref{eq:ss} is rationally degenerate at the $E^2$-sheet.

Our second application is for topological full groups.
The topological full group $F(G)$ of an ample groupoid $G$ encodes the global symmetries of the associated dynamical system, obtained by piecing together the local symmetries encoded in the topological groupoid $G$ we start with.
They were first studied in the setting of Cantor minimal systems~\cite{kri:dim}, and since then they have been generalized in the context of Bratteli diagrams~\cite{gps:full}, shifts of finite type~\cite{matui:fgsft}, self-similar groups~\cite{nekra:fp}, and many other systems, leading to ground-breaking solutions of open problems in geometric group theory~\cites{nekra:int,swz:typef}.
They also offer a dynamical perspective on important objects in group theory, such as the Thompson group $V$ and its variations~\cite{nekra:v}.

Matui's \emph{AH conjecture}~\cite{matui:hkpub}*{Conjecture 2.9} presented a mysterious relation between the groupoid homology of $G$ and the group homology of $F(G)$ in low degrees.
A recent work of Li~\cite{li:algk} gives a satisfactory answer to this problem, by connecting the algebraic $\rK$-theory spectrum of the permutative category that naturally relates $G$ and $F(G)$, denoted $\bK^{\mathrm{alg}}(G)$, to the corresponding homology groups.

The following result, which follows from a combination of Li's results and Theorem~\ref{thm:main-thm}, relates the algebraic $\rK$-theory and operator $\rK$-groups, in a way analogous to Karoubi's conjecture on the algebraic $\rK$-theory of C$^*$-algebras (proved by Suslin and Wodzicki~\cite{suswod:karoubi}).

\begin{theoremA}[Corollary~\ref{cor:ktop-kalg}]
Let $G$ be a groupoid satisfying the assumptions in Theorem~\ref{thm:main-thm}. There is an isomorphism of vector spaces
\[
\rK_i(C^*_r G) \otimes \dQ \cong  \bigoplus_{k\in \dZ}\pi_{i+2k} (\bK^{\mathrm{alg}}(G))\otimes \dQ.
\]
\end{theoremA}

Let us give a sketch of the ideas behind the proof of our main theorem.

In the usual paradigm of noncommutative geometry, one would try to find a ``smooth subalgebra'' $\mathcal{A} \subset C^*_r G$ which gives $\rK_i(\mathcal{A}) \cong \rK_i(C^*_r G)$ on one hand, while having an interesting periodic cyclic homology $\mathrm{HP}_i(\mathcal{A})$ (say, close to the right hand side of Theorem~\ref{thm:main-thm}) on the other, so that the Chern--Connes character map $\rK_i(\mathcal{A}) \to \mathrm{HP}_i(\mathcal{A})$ would give an interesting map on $\rK_i(C^*_r G)$.

Since these are not yet well understood in the setting of ample groupoids, we instead exploit homotopy theory and the recent $\infty$-categorical viewpoint on equivariant $\KK$-theory to ``assemble'' the classical Chern character for the spaces involved in the left hand side.
It has the dual advantage of directly targeting groupoid homology and producing a rational isomorphism essentially by construction.

The first task, which is more formal, is to understand the rationalization of the generalized homology theories involved.
At the level of spectra representing these theories, this is encoded by the stable equivalence of spectra
\begin{equation}\label{eq:chern-intro}
\KU\wedge \bH\dQ \cong \bH\dQ [u,u^{-1}],
\end{equation}
with a formal variable $u$ of degree $2$, comparing the complex topological $\rK$-theory spectrum $\KU$ and the $2$-periodic Eilenberg--MacLane spectrum of the rational numbers.

Under the Dold--Kan correspondence, the (rational) groupoid homology can be interpreted as the homotopy colimit of the diagram of Eilenberg--MacLane spectra for the function spaces $C_c(G^{(n)}, \dQ)$.
Thus, our desired isomorphism should be regarded as a map
\begin{equation}\label{eq:cand-chern}
\rK_i^{\mathrm{top}}(G; C_0(G^{(0)})) \otimes \dQ \to \pi_i\Bigl(\hocolim_{\Delta^\opo} \bH C_c(G^{(\bullet)},\dQ[u,u^{-1}])\Bigr).
\end{equation}
This points to an interpretation of the left-hand side above as a homotopy colimit as well.
This leads to the second task, which is to understand the spaces involved in the left hand side $\rK_*^{\mathrm{top}}(G; C_0(G^{(0)}))$.

Thanks to the torsion-freeness assumption, (at least morally) these should correspond to the classifying space of $G$.
This can be made precise following~\citelist{\cite{py:part-one}\cite{val:kthpgrp}}: there is a certain $G$-C$^*$-algebra $P(C_0(G^{(0)}))$ given as a suitable \emph{sequential} homotopy colimit of $G$-C$^*$-algebras we can build from the simplicial system $(C_0(G^{(n+1)}))_{n = 0}^\infty$, following the recipe of~\citelist{\cite{nestmeyer:loc}\cite{mey:two}} in the triangulated category $\KK^G$.
Then $\rK_i(P(C_0(G^{(0)})) \rtimes G)$ becomes a model of $\rK_i^{\mathrm{top}}(G; C_0(G^{(0)}))$.

The comparison~\eqref{eq:cand-chern} would follow if we could present this sequential colimit as a homotopy colimit over the simplex category $\Delta^\opo$.
However, homotopy colimits of general diagrams in a triangulated category is a delicate problem which requires analysis of higher Toda brackets~\cite{bn:tot}.

We work around this problem by promoting $\KK^G$ to a stable $\infty$-category, following the blueprint provided in~\citelist{\cite{landnik:comp}\cite{bel:kkstable}} for ordinary and equivariant $\KK$-theory for discrete groups.
This allows us to make sense of the $\infty$-colimit of the $G$-simplicial object $C_0(G^{(\bullet+1)})$ in the stable $\infty$-category $\icKK^G$. 

In fact, this presentation applies to the more general case where $G$ has torsion in its isotropy, and is acting on a (separable) C$^*$-algebra $A$. In this case $P(A)$ is the $\infty$-colimit of a simplicial system whose building blocks are induced objects of the form $\Ind_H^G(A)$ for proper open subgroupoids $H\subseteq G$.
In the case of discrete groups, this is analogous to the Davis--Lück presentation of the left hand side of the Baum--Connes conjecture as a homotopy colimit over the orbit category~\cite{dl:bc}. However, since a theory of Bredon homology for groupoids is currently not available, the rest of the argument assumes $G$ is torsion-free, which implies we can work with the simpler $G$-simplicial system $C_0(G^{(\bullet+1)})$.

Using the Morita equivalence between the space $G^{(n)}$ as a groupoid and the transformation groupoid $G \ltimes G^{(n + 1)}$, and by comparing homotopy colimits in $\KK^G$ with $\infty$-colimits in $\icKK^G$, we see that the simplicial system of map-spectra $\map_{\icKK}(\dC,C_0(G^{(\bullet)}))$ leads to the left-hand side in~\eqref{eq:cand-chern} via homotopy colimit.
Then the problem is reduced to finding an isomorphism
\begin{equation}\label{eq:hl-chern}
\hocolim_{\Delta^\opo} \map_{\icKK}(\dC,C_0(G^{(\bullet)}))\wedge \bH\dQ \to \hocolim_{\Delta^\opo} \bH C_c(G^{(\bullet)},\dQ[u,u^{-1}]).
\end{equation}
Now, the left hand side comes from a functor defined in terms of correspondences of spaces and associated Hilbert modules (which we call \emph{Atiyah transfer}), while the right hand side comes from ``summation over fibers'' induced by étale maps, which act on compactly supported function spectra.

To carry out the above comparison, we look at the $\infty$-functors from an $\infty$-category of totally disconnected spaces to the semiadditive $\infty$-category of spectra.
To be precise, let us denote by $\Span_{\pret}(\TDLC)$ the $(2,1)$-category of spans of totally disconnected, locally compact, Hausdorff spaces $X \leftarrow Z \to Y$, with left leg a proper map, and right leg an étale map.
We then have the following intermediate result, which is of its own interest.

\begin{theoremA}[Theorem~\ref{thm:additive-infty-ftrs-on-tdlc} and Proposition~\ref{prop:comp-supp-cont-functors-as-Lan}]\label{thm:intro-thm-D}
Let $\cC$ be a semiadditive $(\infty,1)$-category.
The $\infty$-functors
\[
F\colon \Span_{\pret}(\TDLC) \to \cC
\]
which are product-preserving, compatible with finite quotients, and compactly supported, can be classified by the value of $F$ at the singleton space.
\end{theoremA}

This is comparable to the fact that the span category of finite sets $\Span(\FinSets)$, which naturally sits as a full subcategory of $\Span_{\pret}(\TDLC)$, is the universal singly generated semiadditive $(\infty,1)$-category.

Having such classification at hand, the desired comparison~\eqref{eq:hl-chern} essentially reduces to~\eqref{eq:chern-intro}, combined with the well-established identification $\map_{\icKK}(\dC,\dC)\cong \KU$ (see Section~\ref{subsec:kthspec} for more details). 

\medskip

While our primary results addresses problems originating in operator algebras and topological dynamics, the underlying proof techniques are largely drawn from algebraic topology and homotopy theory.
This naturally invites the question of what can be done beyond the setting of totally disconnected spaces and ample groupoids.
In fact, most of the above ingredients continue to make sense in straightforward ways, but with several interesting twists.

First, the groupoid homology $\rH_i(G; \dQ)$ will take integer grading $i \in \dZ$, reflecting the nontrivial cohomological structure in the ``space'' direction.
Second, the function spectra $\Map_c(Y, \bE)$ should be considered as \emph{higher cosheaves} as developed in~\cite{lur:ht}, subject to the six-functor formalism of higher sheaves.
We will expand on these ideas in our upcoming work~\cite{py:upcoming}, extending our main results to the non-ample case.
Nevertheless, we believe an important aspect of the present paper is that the study of ample groupoids is possible ``just with algebraic topology'', without relying on the theory of higher sheaves, thanks to the zero-dimensional nature of the building blocks.

\medskip

Let us describe the structure of our paper.
In Section~\ref{sec:prelim}, we set our conventions on homotopy theoretic and $\KK$-theoretic constructions, together with some basic material on groupoid homology for ample groupoids.
In Section~\ref{sec:span-tot-disc}, we analyze categories of spans of totally disconnected spaces, and functors out of such categories.
In Section~\ref{sec:k-theory-transfer}, we look at the transfer maps from the bivariant K-theory of C$^*$-algebras.
In Section~\ref{sec:bc-conj}, we review the formulation of the Baum--Connes conjecture via localization of categories~\cites{nestmeyer:loc,mey:two,val:kthpgrp}, and relate it to the stable $\infty$-category $\icKK^G$.
We complete the proof of our main result following the strategy sketched above in Section~\ref{sec:grchern}.
We then discuss the applications and connections to more general settings in Section~\ref{sec:appl}.
In Appendix~\ref{sec:kkg-stable}, we review some categorical constructions around groupoid equivariant $\KK$-theory, and prove that $\icKK^G$ is a stable $\infty$-category when $G$ is a second countable, locally compact, Hausdorff groupoid with a Haar system.
In Appendix~\ref{sec:transfer-maps}, we give a detailed construction of Atiyah transfer for equivariant $\KK$-theory.

\subsection*{Acknowledgements}
We would like to thank B.~Dünzinger and U.~Bunke for useful conversations on equivariant $\rK$-theory and transfer, and A.~Krause for an inspiring discussion leading to the current exposition of Section~\ref{sec:span-tot-disc}.
We would like to thank S.~Bhattacharjee, C.~Bruce, M.~Goffeng, Y.~Kubota, M.~Land, K.~Li, A.~Miller, M.~Whittaker, R.~Willett for stimulating and encouraging conversations.

V.P. acknowledges support from Marie Skłodowska-Curie Postdoctoral Fellowship (Horizon Europe, European Commission, Project No. 101063362).
M.Y.: this research was funded, in part, by The Research Council of Norway [project 324944].

\section{Preliminaries}\label{sec:prelim}

Let us first review the fundamental concepts used throughout the article.
We start with a review of concepts from algebraic topology and homotopy theory in Sections~\ref{subsec:inftycat}--\ref{subsec:spec} (the reader with background in these subjects can safely skip these sections).
We introduce the relevant models for topological $\rK$-theory in Sections~\ref{subsec:kthspec} and~\ref{subsec:cscsec}.
While we opted for an exposition based on sequential topological spectra to make it accessible for operator algebraists, it should not be difficult to read the paper with only a model-independent picture in mind.
In Sections~\ref{sec:grpds-c-star-algs} and~\ref{subsec:correspondences} we review constructions from topological groupoids and C$^*$-algebras, which can be safely skipped by operator algebraists.
Finally, in Section~\ref{sec:gr-hom} we review the Crainic--Moerdijk homology adapted to the setting of ample groupoids.

\subsection{Simplicial homotopy and $\infty$-categories}\label{subsec:inftycat}

Our references for $\infty$-categories are~\cites{lur:ht,lur:ha,land:inftybook}.
Let us fix conventions from simplicial homotopy theory.
We denote the simplex category, i.e., the category of nonempty ordered finite sets, by $\Delta$.
A simplicial set is a contravariant functor $X \colon \Delta \to \Sets$.
We denote the category of simplicial sets by $\sSet$.
The set $\{0, 1, \dots, n \}$ with natural order is denoted by $[n]$, and the corresponding simplicial set is denoted by $\Delta[n]$.
We freely identify the elements of $X_n = X([n])$ with the maps of simplicial sets $\Delta[n] \to X$.
The standard $n$-simplex, which can be considered as the thin geometric realization of $\Delta[n]$, is denoted by $\Delta^n$.

For $0 \le k \le n$, we denote the $k$-th horn in $\Delta[n]$ as $\Lambda^k[n]$.
By definition, it is is the simplicial set obtained from the boundary of the $n$-simplex $\Delta^n$ by discarding its $k$-th face.
A \emph{Kan complex} is a simplicial set $X$ such that any map of simplicial sets $\Lambda^k[n] \to X$ extends to a map $\Delta[n] \to X$, i.e., an element of $X_n$.

We consider $(\infty, 1)$-categories modeled by \emph{quasicategories}, i.e., simplicial sets $\cC$ satisfying the weak Kan condition: any map of simplicial sets $\Lambda^k[n] \to \cC$ for $0 < k < n$ extends to a map $\Delta[n] \to \cC$.
If $\cC$ is a usual category, its nerve $N(\cC)$ is a simplicial set which is a quasicategory.
In the rest of the paper, when it is clear from context that we are viewing ordinary categories as $\infty$-categories through the nerve functor $N$, we will omit writing $N$ from the notation.

We also consider $(2, 1)$-categories.
In this setting, when we want to be clear we write $\circ_1$ for the composition of $1$-cells, and $\circ_2$ for the composition of $2$-cells. Note that $\circ_1$ is a bifunctor, so that that $(\psi \circ_1 \xi)$ makes sense as a $2$-cell when $\psi$ and $\xi$ are $2$-cells. Moreover, $\phi \circ_2 (\psi \circ_1 \xi)$ is a $2$-cell from a $1$-cell $f \circ_1 g$ to another $h$ when we have $2$-cells $\phi \colon f' \circ_1 g' \to h$, $\psi \colon f \to f'$, and $\xi \colon g \to g'$.

Given such a $(2,1)$-category $\cC$, the \emph{Duskin nerve} construction $\ND(\cC)$~\citelist{\cite{MR1897816}\cite{lurie-kerodon}*{Section 2.3}} gives the corresponding quasicategory.
Concretely, the $n$-th set $\ND_n(\cC)$ is given by the set of lax monoidal functors $F$ from  $[n]$ (ordered set regarded as a category) to $\cC$ such that $F(\id_k) = \id_{F(k)}$ (these are also called strictly unitary lax monoidal functors).

\subsection{Spans}

A \emph{span} of finite sets is given by a diagram $X \leftarrow Z \rightarrow Y$ in $\FinSets$, the category of finite sets.
We get a $(2, 1)$-category $\Span(\FinSets)$, whose $0$-cells are finite sets, $1$-cells are spans, and $2$-cells are isomorphisms of spans, with the composition of $1$-cells given by pullbacks.

An $(\infty, 1)$-category is called \emph{semiadditive} if there is a zero object $0$, both coproduct $X \vee Y$ and product $X \times Y$ exist, and the natural map $X \vee Y \to X \times Y$ is a natural isomorphism of these bifunctors.

The $(2, 1)$-category $\Span(\FinSets)$ is semiadditive, since the empty set $\emptyset$ is a zero object and disjoint union $X \coprod Y$ is a model for both coproduct and product.
Moreover, it is a universal singly generated semiadditive $(\infty, 1)$-category in the sense that, if $\cC$ is a semiadditive $(\infty, 1)$-category, then the map
\begin{equation}\label{eq:semiadditive-cat-functor-eval}
\Fun^\times(\Span(\FinSets), \cC) \to \cC, \quad F \mapsto F(*),
\end{equation}
where $\Fun^\times$ denotes the subspace of product-preserving maps of quasicategories, is a weak equivalence of simplicial sets, see for example~\cite{harp:amb}*{Theorem 4.1}.

\subsection{Localization}

Let $\cC$ be a quasicategory with a collection of morphisms $\cW \subset \cC_1$.
The localization of $\cC$ at $\cW$~\citelist{\cite{lur:ha}*{Definition 1.3.4.1}\cite{land:inftybook}*{Chapter 2}} is given by a map of quasicategories $F\colon \cC\to \cC[\cW^{-1}]$ with the following universal property: for every quasicategory $\cD$, the map $\Fun(\cC[\cW^{-1}],\cD) \to \Fun(\cC,\cD)$ induced by $F$ is a fully faithful embedding whose essential image is the collection of maps $G \colon \cC\to \cD$ such that for each $f \in \cW$, the morphism $G(f)$ admits an inverse up to homotopy.
This means that there are morphisms $f'$ and $f''$ in $\cD$ such that $(f', G(f))$ and $(G(f), f'')$ are inner horns which can be filled to a $2$-cell with identity morphisms on the other edges.

If $\cC$ is (the nerve of) an ordinary category, we call this procedure ``$\infty$-localization'' to distinguish it from the localization as an ordinary category, which corresponds to $\pi_0(N(\cC)[\cW^{-1}])$.

\subsection{Homotopy colimits and higher colimits}

Next let us review $\infty$-categorical colimits and the relation to homotopy colimits in model categories.

Given two simplicial sets $X$ and $Y$, their join $X * Y$ is given by
\[
(X * Y)_n = X_n \amalg Y_n \amalg \left( \coprod_{i + j = n - 1} X_i \times Y_j \right) = \coprod_{i = -1}^n X([i]) \times Y(\{i + 1, \dots, n \}),
\]
with the convention that $X(\emptyset) = *$.
Given a map of quasicategories $F \colon \cD \to \cC$, the \emph{slice category} $\cC_{F/}$ is the quasicategory defined by
\[
(\cC_{F/})_n = \Mor_F(\cD * \Delta[n], \cC),
\]
where the subscript $F$ means the maps restricting to $F$ on $\cD$.

\begin{definition}
Let $F\colon \cD\to \cC$ a morphism of quasicategories.
We say that a vertex $G \in \cC$ is a \emph{colimit} of $F$ if it satisfies the following property: there exists a \emph{cone} $\bar{F} \colon \cD\star \Delta[0]\to \cC$ with $\bar{F}(\infty)=G$ such that  $\bar{F}$ is an initial object of the slice category $\cC_{F/}$.
\end{definition}

Here, $\infty$ denotes the unique $0$-cell of $\Delta[0]$, and $\bar{F}$ being initial is defined by asking the Kan complexes $\Mor(\bar{F}, F')$ for $F' \in \cC_{F/}$ to be contractible.
For other characterizations, see~\cite{land:inftybook}*{Section 4.3}.
To emphasize the $\infty$-categorical nature of this definition (particularly when $F$ comes from a functor between usual categories), we sometimes refer to the above concept as \emph{$\infty$-colimit}.

To compare homotopy colimits and $\infty$-colimits, it is technically convenient to consider \emph{combinatorial} model categories in the sense of Smith, see~\cite{dug:thm}.
This condition is satisfied by $\sSet$, and categories of spectra we consider that are modeled by simplicial sets. While the category of topological spaces is not combinatorial, it is still closely related to $\sSet$ by the singular complex functor and the thin geometric realization functor (these form a Quillen equivalence pair). Under this condition, the functor categories $\Fun(\cD, \cC)$ from categories $\cD$ admit model category structures again. We are interested in the \emph{projective model structure}, where weak equivalences and fibrations are objectwise such transformations, and cofibrations are characterized by the lifting property (see~\cite{hirsch:model}*{Section 6.11}).

\begin{definition}
Let $F \colon \cD \to \cC$ be a functor from a category $\cD$ to a (combinatorial) model category $\cC$.
The \emph{homotopy colimit} $\hocolim_\cD F$ of $F$ is given as $\colim_\cD F^c$, where $F^c \to F$ is a cofibrant replacement of $F$ in $\Fun(\cD, \cC)$ for the projective model structure.
\end{definition}

For the model category $\cC$, the $\infty$-colimit in the $\infty$-localization $N(\cC)[\cW^{-1}]$ agrees with the homotopy colimit in $\cC$~\cite{lur:ha}*{Proposition 1.3.4.24}.

\subsection{Spectra}

Let us consider the suspension functor $\Sigma X = S^1 \wedge X$ on $\sSet_*$, where $S^1$ denotes the minimal simplicial circle $\Delta[1] / \Delta[0]$.
A \emph{sequential prespectrum of simplicial sets} is an $\dN$-graded pointed simplicial set equipped with morphisms $\alpha_n\colon \Sigma X_n\to X_{n+1}$.
A morphism $f\colon X_\bullet \to Y_\bullet$ between such objects is a sequence of morphisms $f_n\colon X_n\to Y_n$ of pointed simplicial sets, such that the evident diagram involving $\alpha^X_n$ and $\alpha^Y_n$ is commutative.
This is a simplicially enriched category in the following way: for any simplicial set $K$ and prespectrum $X_\bullet$, define a new prespectrum by assigning $n$ to $X_n\wedge K_+$ with structure maps given by $\alpha_n\wedge \id_K$.
Then the $\Hom$-object $[X,Y]\in \sSet$ is such that $[X,Y]_n$ is the set of prespectra morphisms $X\wedge \Delta[n]_+ \to Y$.
We denote this category by $\Spec$.

Denoting the thin geometric realization of $X \in \sSet$ by $\gr{X}$, the \emph{stable homotopy groups} of a prespectrum $X_\bullet$ are
\begin{equation*}
  \pi_n(X) = \varinjlim_{k\to\infty} \pi_{n+k}(\gr{X_k}) \quad (n \in \dZ).
\end{equation*}
Now, the \emph{Bousfield--Friedlander model structure}~\cite{bf:spectra} on $\Spec$ is defined by taking as weak equivalences the \emph{stable weak equivalences}: those maps $f\colon X_\bullet \to Y_\bullet$ whose image under the stable homotopy groups functor $\pi_*(f)\colon \pi_*(X)\to\pi_*(Y)$ is an isomorphism; the cofibrations are the \emph{stable cofibrations}: those maps $f\colon X\to Y$ such that $f_0$ and all pushout products between $f_n$ and $\alpha_n^Y$
\begin{equation*}
  X_{n+1}\bigsqcup _{\Sigma X_n} \Sigma Y_n \to Y_{n+1}
\end{equation*}
are monomorphisms (cofibrations of simplicial sets).

The fibrant objects in this model structure are precisely the \emph{$\Omega$-spectrum} objects consisting of (pointed) Kan complexes, i.e., those prespectra with values in Kan complexes for which $\gr{\alpha_n}$ is adjunct to a map $\beta_n \colon \gr{X_n}\to \Omega \gr{X_{n+1}}$ which is a weak homotopy equivalence.
A cofibrant object in $\Spec$ is precisely a prespectrum whose structure maps are levelwise injective maps.

We also consider sequential prespectra of topological spaces.
They are given by a sequence of (pointed) topological spaces $(X_n)_{n = 0}^\infty$, endowed with continuous maps
\[
\alpha_n \colon S^1 \wedge X_n \to X_{n + 1}.
\]
The standard model structure of topological spaces can be used to define a model category $\TopCGSpec$ of prespectra of \emph{compactly generated} topological spaces.
When we have a general prespectrum of topological spaces $(X_n)_n$, we apply the $k$-ification functor levelwise to obtain an object in $\TopCGSpec$.
This does not affect the computation of stable homotopy groups, and we tacitly use this identification throughout the paper; for this reason in the sequel we drop the subscript and write $\TopSpec$ instead for simplicity.

There is a Quillen equivalence $\Spec \rightleftarrows \TopSpec$ with right adjoint as follows~\cite{bf:spectra}*{Section 2.5}. For a space $X$, we define its \emph{singular simplicial complex} as the simplicial set 
\[
\Sing X\colon [k]\mapsto \Hom_{\Top}(\Delta^k,X).
\]
This is always a prespectrum of Kan complexes, hence we obtain an $\Omega$-spectrum of simplicial sets if $X_\bullet$ is an $\Omega$-spectrum ($\beta_n \colon X_n \to \Omega X_{n + 1}$ is a weak homotopy equivalence). Levelwise geometric realization is left Quillen adjoint to $\Sing$.

Given two spectra $X$ and $Y$ in either of the above frameworks, define the \emph{smash product} $X \wedge Y$ by
\begin{align*}
(X \wedge Y)_{2n} &= X_n \wedge Y_n,&
(X \wedge Y)_{2n + 1} &= S^1 \wedge X_n \wedge Y_n,
\end{align*}
and the appropriate structure maps, see for example~\cite{adams:shgh}*{Section~II.4}.
In the case of $\TopSpec$, when $Y$ is an $\Omega$-spectrum of CW-complexes, this construction preserves weak equivalence for $X$, see for example~\cite{MR1806878}*{Proposition~11.7}.
This will be the case for all of our examples (up to the above Quillen equivalence in case of $\Spec$). It is well-known that this smash product lacks some desirable features (i.e., strict associativity), and this problem is fixed in modern approaches to spectra, nevertheless the naive definition outlined above will be sufficient for our purposes.

\subsection{Eilenberg--MacLane spectra and rationalization}

Let $A$ be a commutative group.
Then we can view $A$ canonically as a \emph{$\Gamma$-set}, namely a functor from (pointed) finite sets (and pointed maps) to the category of sets.
The $\Gamma$-set $A$ associates $A(F) = A^F$ (the set of pointed maps $F\to A$) to any finite set $F$, and a map $F\to F^\prime$ gives rise to a map $A(F)\to A(F^\prime)$ defined by summation along the fibers of non-basepoint elements.

Following~\cite{bf:spectra}, we extend $A$ to all pointed sets by taking the colimit over finite pointed sets.
By composition, $A$ sends a simplicial set to a simplicial set.
Then, by considering the simplicial spheres $(S^n)_{n\in\dN}$, we obtain a sequence of pointed simplicial sets $(AS^n)_n$, which can be checked to yield a sequential spectrum (see~\cite{bf:spectra}*{Section 4} for more details). 

\begin{definition}
The \emph{Eilenberg--MacLane spectrum} $\bH A$ is the sequential spectrum $(AS^n)_n$.
\end{definition} 

The spectrum $\bH A$ is fibrant and cofibrant in the model category $\Spec$, and in particular, it is an $\Omega$-spectrum~\cite{bf:spectra}*{Theorem 4.2}.
The above construction is natural in $A$, and we consider it a functor
\[
\bH \colon \Ab \to \Spec.
\]
In the sequel we will not distinguish between $\bH A$ and its geometric realization.

Given a spectrum $\bE$, the \emph{rationalization} (or $\dQ$-localization) of $\bE$ is defined as the smash product $\bE_\dQ = \bE\wedge \bH\dQ$ (it is customary to use the Moore spectrum $\bM\dQ$, but we have the equivalence $\bH\dQ \cong \bM\dQ$).
We have a natural isomorphism
\begin{equation}\label{eq:homotop-grp-rationalization}
\pi_n(\bE_\dQ) \cong \pi_n(\bE) \otimes \dQ,
\end{equation}
and the canonical map $\bE \to \bE_\dQ$, after application of $\pi_*$, reads as $x\mapsto x\otimes 1$ through the identification~\eqref{eq:homotop-grp-rationalization}, see~\cite{adams:shgh}*{Proposition II.6.7}. This process makes a generalized homology theory into an ordinary homology theory with coefficients in the following sense.

\begin{proposition}[\cite{adams:shgh}*{Lemma II.6.1}]
Let $\bE$ be an object of $\TopSpec$. Then $\bE\wedge\bH \dZ$ is equivalent to $\bigvee_{n\in\dZ}\Sigma^n\bH G_n$, where $G_n = \rH_n(\bE) = \pi_n(\bE \wedge \bH\dZ)$.
\end{proposition}

As a corollary, there is an equivalence $\bE_\dQ \to \bigvee_{n\in\dN}\Sigma^n \bH(\pi_n(\bE)\otimes \dQ)$.

\begin{definition}[Chern--Dold character]\label{def:chdold}
The equivalence of spectra $\ch_\bE\colon \bE_\dQ \to \bigvee_{n\in\dN}\Sigma^n \bH(\pi_n(\bE)\otimes \dQ)$ is called the \emph{Chern--Dold character}.
\end{definition}

\subsection{Stable \texorpdfstring{$\infty$}{infinity}-categories}\label{subsec:spec} 

We assume $\infty$-categories admit finite limits and colimits. A \emph{square} in $\cC$ is a cone over the morphism $s\colon \Lambda[2]\to \cC$ from a $2$-horn. An initial square is the colimit defining a pushout. There are two outer horns to choose from: in this choice the horn is ``over'' the pushout object. The other choice defines pullbacks as terminal cones which are over the horn.

Let us assume that $\cC$ is \emph{pointed}, that is, there is a zero object, or equivalently a map from the terminal to the initial object.
We say that $\cC$ is \emph{stable} if every square in $\cC$ is a pushout if and only if it is a pullback.
The canonical map from a finite coproduct to a finite product is an equivalence for any stable $\infty$-category, see~\cite{lur:ha}*{Remark 1.1.3.5}.

We consider stable $\infty$-categories as enriched over spectra of simplicial sets~\cite{lur:ha}*{Remark 7.1.2.2}.
The spectrum-enriched map spaces are denoted by $\map_{\cC}(X, Y)$ for $X, Y \in \cC$.

Let $\cW_{\mathrm{stable}}$ denote the class of stable weak equivalences in the simplicial model category $\Spec$. Then the localization $L^H(\Spec,\cW_{\mathrm{stable}})$ is an important example of stable $\infty$-category. We will often write $\Spec$ instead of $L^H(\Spec,\cW_{\mathrm{stable}})$, understanding that we are considering $\Spec$ as a stable $\infty$-category. Similarly, we will view $\TopSpec$ as a stable $\infty$-category and omit localization from the notation. The Quillen equivalence $\Spec \rightleftarrows \TopSpec$ implies an equivalence of the corresponding stable $\infty$-categories~\cite{magee:quillen}.

For $\rK$-theory of C$^*$-algebras, the most fundamental $\infty$-category is the Dwyer--Kan localization $\icKK$ of the category of separable C$^*$-algebras at the class of $\KK$-equivalence homomorphisms~\cite{landnik:comp}, which refines the triangulated structure of $\KK$-theory~\cite{nestmeyer:loc}.
For us, the ``domain'' of our construction is the stable $\infty$-category $\icKK^G$ we construct in Appendix~\ref{sec:kkg-stable}, whose homotopy category is the $G$-equivariant $\KK$-category of separable $G$-C$^*$-algebras.

\subsection{\texorpdfstring{$\rK$}{K}-theory spectra}\label{subsec:kthspec}

The classical description of the topological (complex) $\rK$-theory spectrum, denoted $\KU$, is found for example in~\cite{may:concise}*{p.~208}. We have $\KU_{2n} \cong BU\times \dZ$ and $\KU_{2n+1} \cong U$ for all $n\geq 0$.
The structure maps are given by the Bott equivalence $BU\times \dZ\cong \Omega U$ and the canonical equivalence $U\cong \Omega BU$. 
Setting $\bE=\KU$ in Definition~\ref{def:chdold}, we obtain the classical Chern character equivalence
\begin{equation}\label{eq:chern-equiv}
\KU \wedge \bH\dQ \cong \bH\dQ[u,u^{-1}],
\end{equation}
where we understand $u$ as a formal variable of degree $2$ for notational convenience.

\medskip

Let us recall the construction of the topological spectrum $\bK_n(A)$ for a C$^*$-algebra $A$ from~\cite{dekm:spec}.

Let $\hat{S}$ be C$^*$-algebra $C_0(\dR)$ with the $\dZ/2\dZ$-grading by reflection at the origin.
Denote by $\hat{\cK}$ the algebra of compact operators on the graded Hilbert space $\ell^2(\dN) \oplus \ell^2(\dN)$ with the evident grading, and write $\Cl_n$ for the  complex Clifford algebra with $n$ generators as a $\dZ/2\dZ$-graded C$^*$-algebra.
We then consider the space of (non-unital) $*$-homomorphisms
\begin{equation*}
\bK_n(A) = \Hom(\hat{S}, A \hatotimes \Cl_n \hatotimes \hat{\cK}),
\end{equation*}
equipped with the compact-open topology and basepoint given by the zero map.
These spaces form a spectrum with suitable structure maps described in~\cite{dekm:spec}*{Section 3.3}, satisfying
\begin{equation*}
  \pi_{k+n}(\bK_n(A))\cong \KK_{k+n}(\dC,A\hatotimes \Cl_n \hatotimes \hat{\cK}) \cong 
\rK_{k}(A),
\end{equation*}
for $n\geq 1$ and $k\geq -n$.
Furthermore,~\cite{dekm:spec}*{p.~12} explains that $\bK(A)$ is a positive $\Omega$-spectrum.

\medskip

Let us also give a picture of the $\rK$-theory spectrum based on spaces of Fredholm operators, which may be more familiar to the reader with background in operator algebras. We denote the standard ($\dZ/2\dZ$-graded) Hilbert $A \otimes \Cl_n$-module by $H_n A$.
Then we have the space $\bF_n(A)$ of odd Fredholm endomorphisms $F$ on $H_n A$ satisfying $\norm{F} \le 1$, $F^* = F$, $F^2 - \id \in \cK(H_n A)$, and the anticommutation rule for the Clifford generators.
These spaces form a sequential topological spectrum as in~\cite{bjs:spec}.
Functional calculus defines a weak equivalence $\bF_n(A) \to \bK_n(A)$ (cf.~\cite{dekm:spec}*{Remark 3.2}).

When $X$ is a compact Hausdorff space, ealuation at points of $X$ gives a homeomorphism
\[
\bF_n(C(X)) \to \Map(X, \bF_n(\dC)).
\]
Let $\bF^Q_n$ be the set of unitary elements in the $n$-graded Calkin algebra $Q_n = \cB(H_n \dC) / \cK(H_n \dC)$.
The natural quotient map $\bF_n(\dC) \to \bF^Q_n$ is a homotopy equivalence.
Moreover, Karoubi constructs a homotopy equivalence map $\KU_n \to \bF^Q_n$~\cite{MR261592}*{Corollaire (4.3)}. Combining these, we obtain a zigzag of homotopy equivalences from $\Map(X, \KU_n)$ to $\Map(X, \bF_n(\dC))$.

\begin{corollary}
When $X$ is a compact Hausdorff space, the spectra $\Map(X, \KU_\bullet)$ and $\bK_\bullet(C(X))$ are equivalent as objects of the $\infty$-category $\TopSpec$.
\end{corollary}

We also have a notion of $\rK$-theory spectrum which is purely abstract and is derived from the structure of $\icKK$ as a stable $\infty$-category. The following proposition can be interpreted as stating that $\bK$ provides a ``concrete model'' for the abstract definition. Recall $\gr{-}$ denotes geometric realization. Since stable $\infty$-categories are enriched in spectra, we have a functor
\[
\gr{\map_{\icKK}(\dC,-)}\colon \icKK\to \TopSpec.
\]
On the other hand, since $\bK$ sends $\KK$-equivalence $\ast$-homomorphisms to weak equivalences of spectra, thanks to the universal property of the localization it descends to a well-defined functor
\[
\bK \colon \icKK\to \TopSpec.
\]
The following is from~\cite{landnik:comp}*{Proposition 3.7}.

\begin{proposition}
    There is a natural equivalence of functors $\bK \cong  \gr{\map_{\icKK}(\dC,-)}$. 
\end{proposition}

\subsection{Compactly supported cohomology theories}\label{subsec:cscsec}

Let $X$ be a locally compact Hausdorff space.
Then we have its one-point compactification $X^{+_c} = X \cup \{ \infty \}$, which is a pointed compact Hausdorff space.
Given topological spaces $X$ and $Y$, we will write $\Map(X,Y)$ for the space of continuous maps from $X$ to $Y$. 

Let $\rE^\bullet$ be a generalized cohomology theory given by a sequential (topological) spectrum $(\bE_n)_n$.
That is, we have functors
\[
\rE^n(X)= \varinjlim_k \pi_0\Map(\Sigma^k X, \bE_{n+k})\cong \pi_n \Map(X, \bE_\bullet),
\]
the latter group being the $n$-th (stable) homotopy group of the map spectrum $\Map(X, \bE_\bullet)$.
We define the \emph{compactly supported $\rE$-cohomology} $\rE_c^n(X)$ by
\begin{equation}\label{eq:csc}
\rE_c^n(X) = \varinjlim_k \,[\Sigma^k X^{+_c}, \bE_{n + k}]_*,
\end{equation}
where $[\Sigma^k X^{+_c}, \bE_{n + k}]_*$ denotes the set of homotopy classes of pointed continuous maps from the $k$-fold suspension of $X^{+_c}$ to $\bE_{n + k}$.
When $X$ is already compact, we have $\rE_c^n(X)\cong \rE^n(X)$.

Compactly supported cohomology is contravariant for proper continuous maps: such a map $X \to Y$ induces a well-defined map $X^{+_c} \to Y^{+_c}$. Moreover, we have covariance for embedding of open sets $U \to X$, since we get a map $X^{+_c} \to U^{+_c}$ through the homeomorphism $U^{+_c} \cong X^{+_c}/(U^c \cup \{ \infty \})$.

\begin{example}
Consider the topological $\rK$-theory spectrum $\KU = (\KU_n)_n$. Then the corresponding compactly supported cohomology $\KU_c^n(X)$ agrees with the topological $\rK$-group $\rK^n(X)$, which is by definition the reduced $\rK$-group
\[
\tilde{\rK}^n(X^{+_c}) = \ker ( \rK^n(X^{+_c}) \to \rK^n(\{ \infty \}) )
\]
of the compactification.
In general, when $X$ is a compact space without fixed basepoint, we make sense of $\tilde{\rK}^n(X)$ as the cokernel of the canonical map $\rK^n(*) \to \rK^n(X)$ (the one induced by $X \to *$).
\end{example}

\begin{remark}
Suppose that $X$ is a totally disconnected locally compact Hausdorff space, and that the spectrum $E$ is given by locally contractible spaces $\bE_n$ as a sequential spectrum. Then there is an another model for $\rE_c^n(X)$, given as the inductive limit $\varinjlim_K \rE^n(K)$, where $K$ runs over the directed system of compact open sets in $X$ and inclusion maps between them.

To see this, suppose that we have a class $\alpha$ of $\rE_c^n(X)$ represented by a map $f \colon \Sigma^k X^{+_c} \to \bE_{n + k}$.
Let $\tilde{f}\colon X \to \bE_{n + k}^{S^k}$ be the continuous map corresponding to $S^k \times X^{+_c} \to \bE_{n + k}$ induced by $f$, under the exponential law for spaces.
Take a contractible neighborhood $U \subset \bE_{n + k}$ of the basepoint, and let $\tilde{U} \subset \bE_{n + k}^{S^k}$ be the subset consisting of the maps whose range is contained in $U$.
Then $\tilde{U}$ is an open neighborhood of $\tilde{f}(\infty)$, and the inverse image $\tilde{f}^{-1}(\tilde{U}) \subset X^{+_c}$ contains a compact open neighborhood $V$ of $\infty$.
By construction $K = X^{+_c} \setminus V$ is also a compact open set, and the restriction of $f$ to $K$ defines a class $\alpha' \in \rE^n(K)$, which maps to $\alpha$. This shows that the natural map $\varinjlim_K \rE^n(K) \to \rE_c^n(X)$ is surjective, and a similar argument for homotopies shows that it is injective.
\end{remark}

Let $X$ be a compact Hausdorff space. Let us write $\underline{A}$ for the abelian group $A$ viewed as a constant sheaf on $X$. We have the classical isomorphism induced by the Chern character:
\begin{equation}\label{eq:ch}
\tilde{\rK}^n(X) \otimes \dQ \cong \bigoplus_{k \in \dZ} \tilde{\rH}^{n + 2 k}(X; \underline{\dZ}) \otimes \dQ \cong \bigoplus_{k \in \dZ} \tilde{\rH}^{n + 2 k}(X; \underline{\dQ}),
\end{equation}
where $\tilde{\rH}^\bullet$ is the reduced version of Čech cohomology, defined as the cokernel for $\rH^\bullet(*) \to \rH^\bullet(X)$, see for example~\cite{karoubi-cartan-seminar-16}.
Note that, when $X$ is compact (more generally, when it is paracompact), the Čech cohomology group $\rH^n(X; F)$ with coefficient in a sheaf $F$ agrees with the value of the $n$-th derived functor of the global section functor on $F$, see~\cite{god:fais}*{Section II.5.10}.
            
From~\eqref{eq:chern-equiv}, the Chern character induces an isomorphism
\begin{equation}\label{eq:chdoldconc}
\ch_{\KU,\bullet}\colon\KU_\dQ^\bullet(X)\to \bigoplus_{k \in \dZ} \bH\dQ^{\bullet + 2 k}(X).
\end{equation} 
The main result in~\cite{huber:cechcoh} states that $\bH A$ represents Čech cohomology with coefficients in $\underline{A}$: for a compact Hausdorff space $X$ we have an isomorphism of abelian groups $\rH^n(X,\underline{A})\cong \varinjlim_k [\Sigma^k X, \bH A_{n+k}]$. Thus~\eqref{eq:chdoldconc} is simply the unreduced version of the isomorphism in~\eqref{eq:ch}. We see from~\eqref{eq:csc} that the Chern--Dold character is an isomorphism for the compactly supported cohomology groups as well. It is proved in~\cite{dold:ch} that $\ch_{\KU,\bullet}$ coincides with the classical Chern character.

\subsection{Topological groupoids and C\texorpdfstring{$^*$}{*}-algebras}\label{sec:grpds-c-star-algs}

For an introduction to topological groupoids and their associated C$^*$-algebras we refer to~\cites{ren:group,dw:toolkit}. We do not recall these notions in detail here and generally follow the treatment in our previous work~\cite{py:part-one}. Unless otherwise specified, $G$ will be a second countable, locally compact, Hausdorff groupoid with unit space $G^{(0)}=X$. In general, $G^{(n)}$ will denote the space of $n$-tuples from $G$ whose elements $(\gamma_1,\dots,\gamma_n)$ are composable, i.e., $r(\gamma_{i+1})=s(\gamma_{i})$.

We are going to assume that $G$ is \emph{étale}. This means the source and range maps of $G$, denoted $s$ and $r$, are local homeomorphisms and each point in $G$ admits a local base of \emph{bisections}.
These are open sets $U\subset G$ such that $r|_U$ and $s|_U$ are homeomorphism onto an open set of $X$.

\begin{definition}
An étale groupoid $G$ is said to be \emph{ample} if it admits a topological basis consisting of compact open bisections.
Equivalently, $G$ is ample if its unit space is totally disconnected.
\end{definition}

For any point $x\in X$, we denote by $G_x$ (respectively, $G^x$) the subspace of arrows in $G$ whose source (range) is $x$. For each $x\in X$, define the isotropy (or stabilizer) group $G^x_x = G_x\cap G^x$.

\begin{definition}
We say $G$ is \emph{torsion-free} if the stabilizer group $G^x_x$ is torsion-free for all $x\in X$.
\end{definition}

There are several equivalent notions of \emph{groupoid equivalence} or \emph{Morita equivalence}, see~\cite{fkps:hk}*{Proposition 3.10} for an overview (the notion most familiar to operator algebraists is in~\cite{murewi:morita}); it implies that C$^*$-algebras of equivalent groupoids are strongly Morita equivalent.
A useful example of a groupoid which is equivalent to $G$ is given by the \emph{ampliation} (or \emph{blow-up}) of $G$ associated to an étale surjection $\varphi\colon Y\to X$, as follows:
\begin{equation}\label{eq:ampliation}
G^\varphi= \{ (y_1,\gamma,y_2) \in Y\times G\times Y \mid \varphi(y_1)=r(\gamma), \varphi(y_2)=s(\gamma)\}
\end{equation}
with evident composition and inversion maps. The induced projection $\hat\varphi\colon G^\varphi\to G$ is a prototypical example of Morita equivalence in the sense of~\cite{cramo:hom}*{Section 4.5}.

Since $G$ is étale, the fibers of $r$ and $s$ are discrete and $G$ admits a Haar system given by counting measures.
This allows us to define the reduced groupoid C$^*$-algebra of $G$, which is the completion of the natural convolution algebra $C_c(G)$ of compactly supported continuous functions on $G$ with respect to the operator norm obtained as the supremum over all \emph{regular} representations of $G$ associated to point masses on its unit space $X$. 

Thus, given a (separable) $G$-C$^*$-algebra $A$, for example $C_0(X)$, we can consider the reduced groupoid \emph{crossed product} C$^*$-algebra $A\rtimes_r G$. This is denoted $C^*_r G$ when $A=C_0(X)$.

\subsection{Correspondences}\label{subsec:correspondences}

We consider the $(2,1)$-category $\PrCorr$ of proper correspondences between C$^*$-algebras.
Its  objects ($0$-cells) are given by (separable) C$^*$-algebras, $1$-morphisms ($1$-cells) from $A$ to $B$ are given by Hilbert $B$-modules $E_B$ together with a nondegenerate homomorphism $A \to \cK(E_B)$, and $2$-morphisms ($2$-cells) are given by isomorphisms of such correspondences.
This is a semiadditive $(2,1)$-category, where the direct sum is given by the usual direct sum of C$^*$-algebras.

Given an (ample) groupoid $G$ with base $X$, we will also work with $G$-C$^*$-algebras and equivariant correspondences between them. By definition, a $G$-C$^*$-algebra is a C$^*$-algebra $A$ endowed with a nondegenerate homomorphism of $C_0(X)$ to the center of the multiplier algebra of $A$, together with an action map $C_0(G)_s \otimes_{C_0(X)} A \to C_0(G)_r \otimes_{C_0(X)} A$ satisfying a natural multiplicativity condition (here, the subscripts $s$ and $r$ indicate that we use the source and range maps to define the $X$-algebra structures on $C_0(G)$).
Since we assume that $G$ is ample, we can also encode this structure in terms of the algebra $A_U = C_0(U) A$ for open sets $U \subset X$ and $*$-isomorphisms $A_{s(V)} \to A_{r(V)}$ for open bisections $V \subset G$.

Given a $G$-C$^*$-algebra $A$ and a right Hilbert $A$-module $E$, we can again make sense of the equivariant $G$-action on $E$, for example as a suitable collection of isomorphisms $E_{s(V)} \to E_{r(V)}$ for open bisections $V \subset G$, where we write $E_U = E C_0(U)$ (note that any vector $\xi \in E$ can be written as $\xi = \eta a$ for $\eta \in E$ and $a \in A$, and $\xi f = \eta (a f)$ for $f \in C_0(U)$ is independent of this factorization).
Such an action defines a natural $G$-C$^*$-algebra structure on $\cK_A(E)$.

Then, when $A$ and $B$ are $G$-C$^*$-algebras, an equivariant correspondence is given by an equivariant right Hilbert $B$-module $E$ together with an equivariant nondegenerate $*$-homomorphism $A \to \cK_B(E)$.

Let us also define (proper) groupoid correspondences. Let $G,H$ be étale groupoids. A proper correspondence $\Omega\colon G\to H$ is a space $\Omega$ equipped with a left $G$-action and a right $H$-action, respectively with anchor maps $\rho$ and $\sigma$, such that:
\begin{itemize}
\item the $G$-action commutes with the $H$-action, i.e., $\Omega$ is a \emph{$G$-$H$-bispace};
\item the induced map $\bar{\rho}\colon \Omega/H\to G^{(0)}$ is proper;
\item the right $H$-action is free, proper, and étale (i.e., $\sigma\colon \Omega\to H^{(0)}$ is étale).
\end{itemize}
Note that a Morita equivalence $\Omega\colon G \to H$, as defined in~\cite{murewi:morita}, is a groupoid correspondence where $G$ acts freely and properly, and $\rho, \sigma$ induce homeomorphisms $\bar{\rho}\colon \Omega/H\to G^{(0)}$ and $\bar{\sigma}\colon G\backslash\Omega\to H^{(0)}$.
 
\subsection{Groupoid homology}\label{sec:gr-hom}

Let us briefly review the homology of ample groupoids, with coefficients in sheaves equipped with a groupoid action~\cite{cramo:hom}. When $G$ is a topological groupoid with base space $X$, a $G$-sheaf (of commutative groups) over $X$ consists of a $G$-space $F$ whose structure map $F\to X$ is a surjective local homeomorphism with fibers abelian groups.

In fact, when $G$ is ample, such $G$-sheaves are represented by \emph{unitary $C_c(G; \dZ)$-modules}~\cite{stein:amplesheafmod}. A module $M$ over $C_c(G; \dZ)$ is said to be unitary if it has the factorization property $C_c(G; \dZ) M = M$. The correspondence is given by $\Gamma_c(U; F) = C_c(U; \dZ) M$ for compact open sets $U \subset X$ if $F$ is the sheaf corresponding to such a module $M$.
Below we write $\dZ(U)$ in place of $C_c(U; \dZ)$.

\begin{remark}
Underlying the equivalence between $G$-sheaves and unitary $\dZ G$-modules is the fact that, for any closed set $C\subseteq X$, the functor $\Gamma_c(C,-)$ preserves short exact sequences. Since $\Gamma_c(C,F|_C)\cong \rH^0_c(C,F|_C)$, this follows from the long exact sequence in compactly supported sheaf cohomology, combined with the fact that higher cohomology groups of totally disconnected spaces vanish.
\end{remark}

The definition of groupoid homology is based on \emph{$c$-soft} resolutions of sheaves. A sheaf is $c$-soft if, for any compact subspace $K \subset X$ and $s \in \Gamma(K, F)$, there is a global section $s' \in \Gamma(X, F)$ such that $s'|_K = s$. Any sheaf of commutative groups over a totally disconnected, second countable, locally compact Hausdorff space is $c$-soft~\cite{py:part-one}*{Proposition 2.8}. Thus, when working over ample groupoids, there is no need for sheaf resolutions, which makes it possible to define groupoid homology relying solely on unitary $\dZ G$-modules and homological algebra~\cite{bdgw:matui}*{Section 2}.

\begin{definition}[Ample groupoid homology]
Let $M$ be a unitary $\dZ G$-module. The groupoid homology $\rH_n(G;M)$ of $G$ with coefficients in $M$ is the $n$-th $\Tor$ group $\Tor^{\dZ G}_n(\dZ X, M)$.
\end{definition}

Hence $\rH_n(G;M)$ is characterized as the $n$-th left derived functor of the functor $M\mapsto \dZ G\otimes_{\dZ G} M$, i.e., the group of $G$-coinvariants of $M$~\cite{bdgw:matui}*{Lemma 2.2}.

Given a chain complex of $G$-sheaves $F_\bullet$, we can define the associated \emph{hyperhomology} groups $\dH_*(G;F_\bullet)$ in terms of the corresponding complex of $\dZ G$-modules $M_\bullet$ as follows: $\dH_n(G;F_\bullet)$ is the $n$-th homology group of the total complex coming from the bicomplex $(P_i\otimes_{\dZ G} M_j)_{ij}$, where $P_\bullet\to \dZ X$ is a resolution by projective $\dZ G$-modules. For example, one can take the \emph{flat} resolution induced by the nerve of $G$, given by $P_i=\dZ G^{(i+1)}$, $i\geq 0$ (see~\cite{li:algk}*{Section 2.3} and~\cite{ali:homcor}*{Section 2}).
For concreteness, if $M$ is the $\dZ G$-module associated to the constant sheaf with stalk $A$ (an abelian group), the resulting chain complex from this resolution, which computes $\rH_\ast(G; A)$, is given by $(C_c(G^{(n)},A))_{n = 0}^\infty$ with differential
\begin{align}\label{eq:ghcc}
\partial_n &=\sum_{i=0}^n(-1)^i (d^n_i)_* \colon C_c(G^{(n)}; A) \to C_c(G^{(n-1)}; A),&
(d^n_i)_*(f)(x) &= \sum_{d^n_i(y) = x} f(y),
\end{align}
where $d^n_i\colon G^{(n)}\to G^{(n-1)}$ is the $i$-th face map.

\section{Spans of totally disconnected spaces}\label{sec:span-tot-disc}

We consider the following $(2, 1)$-category $\Span_{\pret}(\TDLC)$.
The $0$-cells (objects) are second countable, locally compact, totally disconnected, Hausdorff topological spaces.
The $1$-cells from $X$ to $Y$ are given by spans $X \leftarrow Z \to Y$, where $Z$ is a second countable, locally compact, totally disconnected, Hausdorff topological space, $Z \to X$ is a proper map and $Z \to Y$ is an étale map.
The $2$-cells are isomorphisms of such spans.
This category is semiadditive, with coproduct (and product) given by disjoint union.

We denote its full subcategory of compact spaces by $\Span_{\pret}(\TDC)$ (properness does not matter in this subcategory).
Note that $\Span(\FinSets)$ is a full subcategory of $\Span_{\pret}(\TDC)$.

Given a semiadditive $(\infty, 1)$-category $\cC$, consider the maps of $(\infty, 1)$-categories from $\Span_{\pret}(\TDC)$ to $\cC$ compatible with direct products, and such that
\begin{equation}\label{eq:continuity-condition-for-finite-quots}
F(X) \cong \colim_{i \in I} F(K_i)
\end{equation}
holds for any projective system of finite sets $(K_i)_{i \in I}$ such that $X \cong \varprojlim_I K_i$ with respect to the maps $F(K_i) \to F(X)$ induced by the maps $X \to K_i$ (note that we have $\TDC^\opo \subset \Span_{\pret}(\TDC)$).
We denote the space of such maps by $\Fun^{\times,c}(\Span_{\pret}(\TDC), \cC)$.

\begin{example}\label{ex:transfer-EM-funcs}
Let $A$ be a commutative group.
Given a locally compact totally disconnected space $X$, we consider the space of compactly supported locally constant $A$-valued functions, $C_c(X; A)$.
If $X \leftarrow Z \to Y$ is a $1$-cell in $\Span_{\pret}(\TDLC)$, we have the induced map $C_c(X; A) \to C_c(Y; A)$ defined as the composition of pullback map $C_c(X; A) \to C_c(Z; A)$ and the summation over fibers $C_c(Z; A) \to C_c(Y; A)$ (the latter is well defined since $Z \to Y$ is an étale map).
Any two $1$-cells connected by a $2$-cell would induce the same map, hence we get a well-defined map of $(\infty, 1)$-categories $\Span_{\pret}(\TDLC) \to \Ab$.
By the functoriality of the Eilenberg--MacLane spectra construction $\Ab \to \Spec$, we get a map of $(\infty, 1)$-categories $\Span_{\pret}(\TDLC) \to \Spec$.
This preserves (co)products.
Moreover, when $X$ is compact, we have $C_c(X; A) = C(X; A)$, and it has the desired continuity property~\eqref{eq:continuity-condition-for-finite-quots}.
This way we get an element in $\Fun^{\times,c}(\Span_{\pret}(\TDC), \Spec)$ given by $X \mapsto \bH C(X; A)$.
\end{example}

\subsection{Classification}

Suppose that $\cC$ is an $(\infty, 1)$-category and $\cC^0 \to \cC$ is the inclusion of a full $(\infty, 1)$-subcategory.
Given another $(\infty, 1)$-category $\cD$, a map $F \in \Fun(\cC, \cD)$ is called a \emph{left Kan extension} of a map $F_0 \in \Fun(\cC^0, \cD)$ if
\[
F(X) \cong \colim_{(Y \to X) \in \cC^0_{/X}} F_0(Y)
\]
holds for all $0$-cell $X$ of $\cC$, see~\cite{lur:ht}*{Section 4.3.2}.
For a fixed $F_0$, the space of its left Kan extensions to $\cC$ is, if nonempty, contractible.

\begin{lemma}\label{lem:overcat-of-2-1-cat}
Let $\cC$ be a $(2, 1)$-category, and $X$ be a $0$-cell of $\cC$.
The overcategory $\cC_{/X}$, as a simplicial set, is represented by the Duskin nerve of the $(2, 1)$-category $\cD$ with the following ingredients:
\begin{itemize}
\item the $0$-cells are the $1$-cells in $\cC$ of the form $f\colon Y \to X$;
\item a $1$-cell from $f \colon Y \to X$ to $g \colon Z \to X$ is given by a $1$-cell $h \colon Y \to Z$ in $\cC$ together with a $2$-cell $\phi \colon g h \to f$; and
\item a $2$-cell from $(h, \phi)$ to $(h', \phi')$ is given by a $2$-cell $\psi \colon h \to h'$ in $\cC$ such that $\phi' \circ_2 (\id_g \circ_1 \psi) = \phi$.
\end{itemize}
\end{lemma}

\begin{proof}
Unpacking the definitions, an $n$-cell of the simplicial set for $\cC_{/X}$ is given by a sequence $Y_0$, \dots, $Y_{n+1}$ of $0$-cells such that $Y_{n+1} = X$, a collection of $1$-cells $f_{i, j} \colon Y_i \to Y_j$ for $i < j$, and $2$-cells $f_{i, j} \circ_1 f_{j, k} \to f_{i, k}$ satisfying the associativity constraint.
On the other hand, an $n$-cell in $\ND(\cD)$ is given by a sequence of $0$-cells of $\cD$, i.e., $Y_0 \to X$, \dots, $Y_n \to X$, and $1$-cells $Y_i \to Y_j$ for $i < j$, and again $2$-cells analogous to above.
We thus have a concrete bijective correspondence of defining data.
\end{proof}

\begin{theorem}\label{thm:additive-infty-ftrs-on-tdlc}
Let $\cC$ be a semiadditive $(\infty, 1)$-category.
The restriction map
\begin{equation}\label{eq:tdc-to-fin-res-map}
\Fun^{\times,c}(\Span_{\pret}(\TDC), \cC) \to \Fun^\times(\Span(\FinSets), \cC)
\end{equation}
has contractible fibers.
\end{theorem}

\begin{proof}
Suppose $X \leftarrow U \to Y$ is a $1$-cell in $\Span_{\pret}(\TDC)$, with $X$ and $Y$ being finite sets.
Since we assume $U \to Y$ to be an étale map, $U$ is also a finite set.
This shows that $\Span(\FinSets)$ is a full $(2,1)$-subcategory of $\Span_{\pret}(\TDC)$.
Let $F_{0}$ be the restriction of $F$ to $\Span(\FinSets)$.
We claim that $F$ is a model of left Kan extension of $F_0$ to $\Span_{\pret}(\TDC)$.

For each totally disconnected compact space $X$, the condition~\eqref{eq:continuity-condition-for-finite-quots} gives a colimit for the slice category $\FinSets^\opo_{/X}$ with respect to the full embedding $\FinSets^\opo \to \TDC^\opo$ (that is, the category $\FinSets^\opo \times_{\TDC^\opo} \TDC^\opo_{/X}$ consisting of the morphisms $Z \to X$ in $\TDC^\opo$ with $Z \in \FinSets$), while the left Kan extension is about the slice category $\Span(\FinSets)_{/X}$.
We thus need to show that the inclusion $\FinSets^\opo_{/X} \to \Span(\FinSets)_{/X}$ is cofinal.
By the $\infty$-categorical version of Quillen's Theorem A~\cite{lur:ht}*{Theorem 4.1.3.1}, it amounts to checking the weak contractibility of the comma $\infty$-category $\cC \times_\cD \cD_{U/}$ for $\cC = \FinSets^\opo_{/X}$, $\cD = \Span(\FinSets)_{/X}$, and a $0$-cell $U$ in $\cD$, that is, a span $Y \leftarrow U \to X$ for $Y \in \FinSets$.

By Lemma~\ref{lem:overcat-of-2-1-cat} and its analogue to undercategories, the simplicial set for $\cC \times_\cD \cD_{U /}$ is the Duskin nerve of the following $(2,1)$-category, which is in fact an ordinary category.
\begin{itemize}
\item the $0$-cells are given by $(f \colon X \to Z, (V_{y, z})_{(y, z) \in Y \times Z}, (\psi_y)_{y \in Y})$, where $f$ is a continuous map to a finite set $Z$ and $V_{y, z}$ are finite sets, isomorphisms of étale spaces $\psi_y\colon U_y \to \coprod_z V_{y, z} \times X_z$ over $X$ for each $y \in Y$;
\item the $1$-cells from from such a $0$-cell $(f \colon X \to Z, V_\bullet, \psi_\bullet)$ to another one $(f'\colon X \to Z', V'_\bullet, \psi'_\bullet)$ are given by the maps $g \colon Z' \to Z$ (morphisms $Z \to Z'$ in $\FinSets^\opo$) such that the partition $(X_{z'})_{z' \in Z'}$ is a refinement of $(X_z)_{z \in Z}$ subject to $g$, together with bijective maps $\phi_{y,z'} \colon V_{y,g(z')} \to V'_{y,z'}$ for $y \in Y$ and $z' \in Z'$, compatible with $\psi_y$ and $\psi'_y$; and
\item there is a $2$-cell from a $1$-cell $(g^1, \phi^1)$ to another $(g^2, \phi^2)$ (which are both from a $0$-cell $(Z, V_{\bullet}, \psi_\bullet)$ to another $(Z', V'_\bullet, \psi'_\bullet)$) only when $g^1 = g^2$ and $\phi^1 = \phi^2$, and the identity $2$-morphism is the only choice.
\end{itemize}

It remains to check that this category is filtered, since such categories indeed have contractible nerve, see~\cite{MR0338129}*{p.~85}.
First, suppose we have two $0$-cells $(Z^1, V^1, \psi^1)$ and $(Z^2, V^2, \psi^2)$, and let us construct a $0$-cell $(Z^3, V^3, \psi^3)$ that receives $1$-cells from both.
As a set, $Z^3$ is the index set for the refinement partition of $(X_{z_1})_{z_1 \in Z^1}$ and $(X_{z_2})_{z_2 \in Z^2}$ (to be precise, it is the index set for nonempty subsets $X' \subset X$ of the form $X' = X_{z_1} \cap X_{z_2}$ for some $z_1 \in Z^1$ and $z_2 \in Z^2$, counted without multiplicity).
We write the corresponding maps as $g^1 \colon Z^3 \to Z^1$ and $g^2 \colon Z^3 \to Z^2$.
Next, given $y \in Y$ and $z_3 \in Z^3$, we fix $x \in X_{z_3}$ and set $V^3_{y, z_3} = U_x$.
We then define $\phi^1_{y, z_3} \colon V^1_{y, g^1(z_3)} \to V^3_{y, z_3}$ by $\phi^1(v) = u$, where $u \in U_x = V^3_{y, z_3}$ is the element that corresponds to $(v, x) \in V^1_{y, g^1(z_3)} \times X_{g^1(z_3)}$ under the map $\psi^1_y$ (note that we have $X_{z_3} \subset X_{g^1(z_3)}$).
The maps $\psi^2_{y, z_2} \colon V^2_{y, g^2(z_3)} \to V^3_{y, z_3}$ are defined
in the same way.

Next, suppose that we have parallel $1$-cells $(g, \phi)$ and $(g', \phi')$ from a $0$-cell $(f\colon X \to Z, V_{\bullet}, \psi_\bullet)$ to another $(f'\colon X \to Z', V'_\bullet, \psi'_\bullet)$, and let us find a $1$-cell $(g'', \phi'') \colon (Z', V'_\bullet, \psi'_\bullet) \to (Z'', V''_\bullet, \psi''_\bullet)$ that coequalizes both.
For this, it is enough to take the image of $f'$ as $Z''$ (that is, the set of $z' \in Z'$ such that $X_{z'} \neq \emptyset$), and take the obvious restrictions $V''$ and $\psi''$.
\end{proof}

This result, combined with~\eqref{eq:semiadditive-cat-functor-eval}, implies that the functors in $\Fun^{\times,c}(\Span_{\pret}(\TDC), \cC)$ are classified by their values on the singleton space.
Moreover, if $\cC$ has directed colimits, then the restriction map~\eqref{eq:tdc-to-fin-res-map} is a weak homotopy equivalence.

Eventually we will consider functors defined on $\Span_{\pret}(\TDLC)$ instead of $\Span_{\pret}(\TDC)$.
Let us denote by $\Fun^{\times,cc}(\Span_{\pret}(\TDLC), \cC)$ the space of functors satisfying the continuity condition~\eqref{eq:continuity-condition-for-finite-quots} for compact spaces $X$, and an additional ``compact support'' condition
\begin{equation}\label{eq:cpt-supp-as-colim}
F(Y) \cong \colim_{X} F(X),
\end{equation}
where $X$ runs over the poset of compact open subspaces of $Y$.

\begin{proposition}\label{prop:comp-supp-cont-functors-as-Lan}
The functors $F \in \Fun^{\times,cc}(\Span_{\pret}(\TDLC), \cC)$ are left Kan extensions of their restrictions to $\Span_{\pret}(\TDC)$.
\end{proposition}

\begin{proof}
Again we need to check that, for each $Y \in \TDLC$, the value $F(Y)$ is a colimit of $F(X)$ for the spans
\begin{equation}\label{eq:span-from-compact}
\begin{tikzcd}
X & U \arrow[l, "p"'] \arrow[r, "f"] & Y
\end{tikzcd}
\end{equation}
in $\Span_{\pret}(\TDLC)$, with $X$ compact.
By the properness assumption on $p$, the intermediate space $U$ is compact, and we get a span $X \leftarrow U \to f(U)$ in $\Span_{\pret}(\TDC)$.
This, together with the inclusion $f(U) \to Y$, factorizes~\eqref{eq:span-from-compact}.
Thus $F(Y)$ is indeed a colimit of $F$ on $\Span_{\pret}(\TDC)_{/Y}$.
\end{proof}

We thus see the functors in $\Fun^{\times,cc}(\Span_{\pret}(\TDLC), \cC)$ are classified by their values on the singleton. 
In particular, for each object $\bE$ in $\TopSpec$ there is an essentially unique functor
\begin{equation}\label{eq:fe}
F^\bE\colon \Span_{\pret}(\TDLC)\to \TopSpec
\end{equation}
such that $F^\bE(\ast)=\bE$.
Concretely, we have $F^\bE(Y) = \Map_c(Y,\bE)$, where the latter stands for a suitably defined compactly-supported mapping spectrum from a locally compact totally disconnected space $Y$ to a CW-spectrum $\bE$.
The homotopy groups of $F^\bE(Y)$ recover the compactly supported $\rE$-cohomology $\rE^n_c(Y)$.
See Section~\ref{subsec:mapcsheaves} for more details, including a sheaf-theoretic interpretation of $F^{\bE}$.

\section{Atiyah transfer}\label{sec:k-theory-transfer}

Let $A$ and $B$ be C$^*$-algebras. A \emph{proper correspondence} from $A$ to $B$ is the data of a right Hilbert $B$-module $E$ together with a $*$-homomorphism $A \to \cK_B(E)$.
Let $X,Y$ be locally compact Hausdorff topological spaces and $f\colon Y\to X$ a surjective local homeomorphism.
We define a proper $C_0(Y)$-$C_0(X)$-correspondence $E_f$ as follows: consider $C_c(Y)$ with the left action of $C_0(Y)$ by multiplication, and right action of $C_0(X)$ given by $\xi\cdot g= \xi \tilde{g}$, where $\tilde{g}$ is the bounded continuous function on $Y$ obtained by lifting $g$.
The right inner product is 
\[
(\xi|\zeta) (x)= \sum_{z\in f^{-1}(x)} \bar{\xi}(z)\zeta(z).
\]
Then $E_f$ is defined as the completion of $C_c(Y)$ with respect to the norm induced by the inner product.
It is a standard fact that the left action just defined is by compact operators. 

\begin{remark}
In~\cite{ati:char} Atiyah defines a \emph{transfer map} when $f\colon Y\to X$ is a finite covering between compact Hausdorff spaces.
In this case, given a vector bundle $E$ over $Y$ and its associated sheaf $\Gamma(E)$, the pushforward sheaf $f_*(\Gamma(E))$ defines a vector bundle $f_*(E)$ on $X$.
At level of fibers, we have $f_*(E)_x = \bigoplus_{f(y)=x} E_y$.
Summation over the fibers gives the ``wrong way map'' $f_!\colon \rK^*(Y)\to \rK^*(X)$ between the respective topological $\rK$-groups.
This coincides with the morphism of $\rK$-groups induced by the class $[E_f]\in \KK(C_0(Y),C_0(X))$.
\end{remark}

The functoriality of the assignment $f\mapsto E_f$ can be understood via by using proper \emph{groupoid correspondences} (see Section~\ref{subsec:correspondences}).
Given $f\colon Y\to X$ as above, the ampliation groupoid $X^f$ as in~\eqref{eq:ampliation} has unit space $Y$ and orbit space canonically homeomorphic to $X$.
The groupoid $X^f$ is free and proper, thus there is a Morita equivalence $\Omega=Y$ with left action of $X^f$ and right action of $X$ (i.e., there is a local homeomorphism $f=\sigma\colon Y \to X$). 

Let us denote by $\Omega_f$ the groupoid correspondence obtained from $\Omega$ by restricting the left action from $X^f$ to its unit space $Y$.
Unpacking the definitions, we obtain the following.

\begin{lemma}\label{lem:grcorr-comp}
Let $f\colon Y\to X$ and $g\colon X\to Z$ be (surjective) local homeomorphisms. There is a natural isomorphism of groupoid correspondences between $\Omega_{gf}$ and 
\[
\Omega_g\circ \Omega_f= \Omega_f \times _X \Omega _g
\]
with anchor maps obtained by composition with the canonical projections.
\end{lemma}

Groupoid correspondences form a $2$-category $\GrCorr$. We are going to consider the subcategory of proper groupoid correspondences $\PrGrCorr$. We denote by $\PrCorr$ the $2$-category of proper C$^*$-algebra correspondences. Thanks to~\cite{akm:grcorr}*{Theorems 7.13 and 7.14}. We have a functor of $(2,1)$-categories 
\[
E\colon\PrGrCorr\to \PrCorr
\]
where $\Omega\colon G\to H$ is sent to the correspondence $E(\Omega)\colon C^*G\to C^*H$. Here we are considering the full C$^*$-algebras, however we will apply this construction to amenable groupoids only so the distinction will not matter (note that proper groupoids are amenable,~\cite{renroch:amgrp}*{Cor.~2.1.17 \& Proposition~6.1.10}).

It follows easily from the construction that $E_f \cong E(\Omega_f)$. From Lemma~\ref{lem:grcorr-comp} we have $E_{gf}\cong E_g\circ E_f$. If $f\colon Y\to X$ is a $G$-equivariant map between free and proper $G$-spaces, $\Omega_f$ carries a $G$-action and it gives a $G$-equivariant correspondence $\Omega^G_f\colon Y\to X$. Applying $E$ to  $\Omega^G_f\colon Y\to X$ yields a $G$-equivariant C$^*$-algebra correspondence
\[
E^G(\Omega^G_f)=E^G_f \colon C_0(Y)\to C_0(X).
\]

Let $X \leftarrow Z \to Y$ be a span in $\Span_{\pret}(\TDLC^G)$, where $\TDLC^G$ is the category of free and proper $G$-spaces which are locally compact, Hausdorff, and totally disconnected, with $G$-equivariant continuous maps as morphisms (this is the $G$-equivariant version of the $(2,1)$-category $\Span_{\pret}(\TDLC)$ introduced in Section~\ref{sec:span-tot-disc}). Let $g$ (respectively, $f$) denote the left (respectively, right) leg in $X \leftarrow Z \to Y$. We define $\Omega^G_Z$ as the modification of $\Omega^G_f$ obtained by considering $g\colon \Omega\to X$ as the left anchor map (recall that $\Omega=Z$ as topological spaces). We obtain a map of $(2,1)$-categories 
\[
\Span_{\pret}(\TDLC^G) \to \PrCorr^G,
\]
sending the span $X \leftarrow Z \to Y$ to $E^G(\Omega^G_Z)=E^G(Z)$. For an object $X\in \TDLC^G$, let us denote by $q_X$ the quotient map $X\to G\backslash X$. Note $G\ltimes X$ is Morita equivalent to $G\backslash X$ through $\Omega^G_{q_X}$. 

\begin{proposition}\label{prop:span-corr-comm}
The following diagram commutes up to $2$-categorical equivalence.
\[
\begin{tikzcd}
\Span_{\pret}(\TDLC^G) \arrow[r, "E^G"] \arrow[d, "G\backslash(-)"] & \PrCorr^G \arrow[d, "- \rtimes G"] \\
\Span_{\pret}(\TDLC) \arrow[r, "E"]                 & \PrCorr               
\end{tikzcd}
\]
More precisely, given  $X \leftarrow Z \to Y$ in $\Span_{\pret}(\TDLC^G)$, there is a natural isomorphism between 
\[
E^G_{q_Y} \circ  E^G(Z)\rtimes G \circ (E^G_{q_X})^{\mathrm{op}}\quad \mathrm{ and } \quad E(G\backslash Z)
\]
as correspondences $C_0(G\backslash X)\to C_0(G\backslash Y)$.
\end{proposition} 

The right vertical map above is defined in analogy with Kasparov's descent functor in $\KK^G$-theory, see for example~\cite{lg:kk}*{Theorem 3.4}. In order to simplify calculations, it is useful to introduce a functor 
\[
G\ltimes (-)\colon \Span_{\pret}(\TDLC^G)\to \PrGrCorr
\]
defined as follows. Given $\Omega^G_Z$, define the underlying topological space for $G\ltimes \Omega^G_Z$ as the fiber product $Z\times_{\rho,r} G$, where $\rho\colon Z\to G^{(0)}$ is the structure map and $r\colon G\to G^{(0)}$ is the range map. Note that $\rho$ factors through $g$ and through $f$. The left anchor map is given by projection on the first coordinate followed by $g$. The right anchor is given by $(z,\gamma)\mapsto \gamma^{-1}\cdot f(z)$. The left and right action are then given by 
\[
(\gamma^\prime,x) \cdot  (z, \gamma)= (\gamma ^\prime \cdot z, \gamma ^\prime\gamma) \qquad\quad (z, \gamma)\cdot (\gamma^\prime,y^\prime)= (z, \gamma \gamma^\prime),
\]
where $(\gamma^\prime,x^\prime )\in G\ltimes X$ has range $x^\prime$, $(\gamma,y)\in G\ltimes Y$ has range $y$, $(z, \gamma)$ satisfies $\rho(z)=r(\gamma)$, $g(z)=x$ (on the left), and $\gamma^{-1}\cdot f(z)=y^\prime$ (on the right).

We have defined a proper groupoid correspondence
\[
G\ltimes \Omega^G_Z\colon G\ltimes X \to G\ltimes Y.
\]

\begin{remark}\label{rem:diagj}
It is not hard to verify that the diagram
\begin{equation}\label{eq:diagj}
\begin{tikzcd}
\Span_{\pret}(\TDLC^G) \arrow[r, "E^G"] \arrow[d, "G\ltimes -"] & \PrCorr^G \arrow[d, "- \rtimes G"] \\
\PrGrCorr \arrow[r, "E"]                 & \PrCorr               
\end{tikzcd}
\end{equation}
commutes up to natural equivalence.
Indeed, if we denote by $\mathfrak E$ the upper semicontinuous bundle associated to a Hilbert $G$-module $E$, then $E\rtimes G$ is a completion of $\Gamma_c(G;r^*\mathfrak{E})$ with inner product
\begin{equation}\label{eq:innprod}
(\xi|\zeta) (\gamma)= \sum_{s(\eta)=r(\gamma)} \eta^{-1}\cdot ({\xi}(\eta)|\zeta(\eta\gamma)).
\end{equation}
For a $G$-span $X \leftarrow Z \to Y$, we have that $E^G(Z)$ is a completion of $\Gamma_c(Z)$, so that $\Gamma_c(G;r^*\mathfrak{E})$ is a completion of $\Gamma_c(Z\times_{\rho,r} G)$ as prescribed by $E(G\ltimes \Omega^G_Z)$. The inner product in~\eqref{eq:innprod} corresponds to
\begin{equation*}
(\xi|\zeta) (\gamma,y)= \sum_{s(\eta)=r(\gamma)} \sum_{f(z)=\eta\cdot y}(\bar{\xi}(z,\eta)|\zeta(z,\eta\gamma)),
\end{equation*}
which can be rewritten as 
\begin{equation}\label{eq:innprod2}
(\xi|\zeta) (\gamma,y)= \sum_{d(x)=y} \bar{\xi}(x)|\zeta(x\cdot (\gamma,y)),
\end{equation}
where $d\colon Z\times_{\rho,r} G\to Y$ is the right anchor map defined above for $G\ltimes \Omega^G_Z$.
Equation~\eqref{eq:innprod2} is the inner product as prescribed by $E(G\ltimes \Omega^G_Z)$, thus the inner products match up. The verification of compatibility for left and right action is proved just as easily, hence passing to completion we have that diagram~\eqref{eq:diagj} does indeed commute up to isomorphism of C$^*$-correspondences.
\end{remark}

\begin{proposition}
There is a well-defined functor $G\ltimes -\colon \Span_{\pret}(\TDLC^G)\to \PrGrCorr$. In particular there are natural isomorphisms 
\[
G\ltimes (Z\circ Z^\prime)\cong G\ltimes Z \circ G\ltimes Z^\prime,
\]
where $Z\circ Z^\prime$ denotes a composition of spans in $\Span_{\pret}(\TDLC^G)$.
\end{proposition}

\begin{proof}[Proof of Proposition~\ref{prop:span-corr-comm}]
The squares in the diagram
\[
\begin{tikzcd}
X  \arrow[d, "q_X"] & Z \arrow[l, "g"'] \arrow[r, "f"] \arrow[d, "q_Z"] & Y \arrow[d, "q_Y"] \\
G\backslash X  & G\backslash Z  \arrow[l, "\bar{g}"'] \arrow[r, "\bar{f}"]                 & G\backslash Y               
\end{tikzcd}
\]
commute.
Note that there is an isomorphism $\Omega^G_Z\cong \Omega^G_f\circ \Omega_G^g$, where $\Omega_G^g\colon X\to Z$ is given by $Z$ with left anchor map $g$ and right anchor map $\id_Z$. By functoriality $ G\ltimes \Omega^G_Z$ is isomorphic to $G\ltimes \Omega^G_f\circ G\ltimes \Omega_G^g$. 
We have an isomorphism $\Omega^{G}_{q_Y}\circ (G\ltimes \Omega^G_f) \cong \Omega^G_{q_Y f} \cong \Omega^G_{\bar{f} q_Z}\cong \Omega_{\bar{f}}\circ \Omega^G_{q_Z}$. Now we have an isomorphism $\Omega^G_{q_Z}\circ (G\ltimes \Omega^g_G)\cong \Res^G_{G\ltimes X}(\Omega^G_{q_Z})$, where the latter correspondence is just $\Omega^G_{q_Z}$ with left action induced along the morphism $G\ltimes X \to G$ with anchor map $g^{(0)}$. Consider the correspondence $\Omega^{\bar{g}}\circ \Omega^G_{q_X}$. Its underlying topological space is the fiber product, denoted $Z_0$, over $X \to G\backslash X \leftarrow G\backslash Z$. By the universal property of pullbacks, we get a map $Z\to Z_0$. This map is compatible with the left and right actions, giving us a map $\Res^G_{G\ltimes X}(\Omega^G_{q_Z})\to \Omega^{\bar{g}}\circ \Omega^G_{q_X}$. This map is an isomorphism because morphisms between principal $G$-bundles over the same base are necessarily isomorphisms. Overall, we have
 \[
\Omega^{G}_{q_Y}\circ (G\ltimes \Omega^G_Z) \cong \Omega^{G}_{q_Y}\circ (G\ltimes \Omega^G_f)\circ (G\ltimes \Omega^g_G) \cong \Omega_{\bar{f}}\circ \Omega^{\bar{g}}\circ \Omega^G_{q_X}\cong \Omega_{G\backslash Z}\circ \Omega^G_{q_X}.
\]
The claim follows applying $E^G$ and using that $E^G(-)\rtimes G \cong E(- \rtimes G)$ (see Remark~\ref{rem:diagj}).
\end{proof}

Let $E$ be a proper correspondence between C$^*$-algebras $A$ and $B$.
Then any $C$-$A$-Kasparov cycle $(H, F)$ gives a $C$-$B$-Kasparov cycle $(H \otimes_A E, F \otimes \id)$, which induces a map $\KK(C, A) \to \KK(C, B)$.

\begin{proposition}\label{prop:simplicially-enriched-model-of-atiyah-transfer}
There is an $\infty$-functor 
\[
T_{\icKK^G}\colon \PrCorr^G \to \icKK^G
\]
which lifts the Atiyah transfer construction.
\end{proposition}

\begin{proof}
In~\cite{arXiv:2508.21601} Meyer proves a theorem establishing that the $(2, 1)$-category $\PrCorr$ is in fact a Dwyer--Kan localization of the category of C$^*$-algebras and $\ast$-homomorphisms with respect to the class of homomorphisms of the form $A \to \cK(A \oplus E)$ for separable right Hilbert $A$-modules $E$.
These maps are $\KK^G$-equivalences, hence a straightforward generalization of his theorem to the $G$-equivariant setting implies the claim.
We present an alternative proof in Section~\ref{app:alt-proof}.
\end{proof}

\subsection{Transfer to the category of spectra}

Recall that the functor $\Map_{\icKK}(\dC, -)\colon \icKK \to \sSet_*$ lifts to a functor $\map_{\icKK}(\dC, -)$ to the category $\Spec$ of spectra~\cite{landnik:comp}*{Lemma 3.6}.
Composing this with the $\infty$-functor $T_\icKK$ from Proposition~\ref{prop:simplicially-enriched-model-of-atiyah-transfer} for the trivial groupoid, and the construction $E_f$, we obtain an $\infty$-functor from the category of locally compact spaces with étale maps to $\Spec$, sending $Y$ to a spectrum $\Phi(Y)$ such that $\Omega^\infty \Phi(Y)$ is weakly equivalent to the Kan complex $\Map_{\icKK}(\dC, C_0(Y))$. In other words, $\Phi$ is defined as

\begin{equation*}
\Span_{\pret}(\TDLC)\overset{E}{\longrightarrow} \PrCorr\xrightarrow{T_\icKK} \icKK \xrightarrow{\map_{\KK}(\dC, -)}\Spec.
\end{equation*}

\begin{proposition}\label{prop:Phi-has-prop-cc}
The $\infty$-functor $\Phi$ satisfies conditions~\eqref{eq:continuity-condition-for-finite-quots} and~\eqref{eq:cpt-supp-as-colim}.
\end{proposition}
\begin{proof}
This follows from the fact that $\map_{\KK}(\dC, -)$ preserves colimits.
\end{proof}

Let $M$ be a commutative group.
We denote by $\Psi_M$ the $\infty$-functor given by Example~\ref{ex:transfer-EM-funcs}. Thus, it is the composition
\[
\begin{tikzcd}
\Span_{\pret}(\TDLC) \arrow[r] & \Ab \arrow[r, "H"] & \Spec,
\end{tikzcd}
\]
where the first step is given by $Y \mapsto C_c(Y, M)$ at the level of $0$-cells.

\begin{proposition}\label{prop:final-transfer}
The rationalization $\Phi(-) \wedge \bH\dQ$ is naturally isomorphic to $\Psi_{\dQ[u, u^{-1}]}$.
\end{proposition}

\begin{proof}
We know that $Y \mapsto \Phi(Y) \wedge \bH\dQ$ is in $\Fun^{\times,c}(\Span_{\pret}(\TDLC), \Spec)$ by Proposition~\ref{prop:Phi-has-prop-cc}.
The $\infty$-functors $\Psi_M$ are also in $\Fun^{\times,c}(\Span_{\pret}(\TDLC), \Spec)$, as conditions~\eqref{eq:continuity-condition-for-finite-quots} and~\eqref{eq:cpt-supp-as-colim} reduce to
\[
C(X, M) = \varinjlim_i C(K_i, M), \quad
C_c(Y, M) = \varinjlim_n C(X_n, M)
\]
for the finite factor spaces $(K_i)_i$ of $X$ and the compact open sets $(X_n)_n$ of $Y$ appearing in these conditions.
Proposition~\ref{prop:comp-supp-cont-functors-as-Lan} and the equivalence in~\eqref{eq:chern-equiv} give the desired comparison.
\end{proof}

Since $\gr{\map_{\KK}(\dC, \dC)}\cong \KU$, the functors $\Phi$ and $\Psi_M$ appearing in Proposition~\ref{prop:final-transfer} can be expressed in terms of $F^\bE$ (after composing with geometric realization) as follows:
\begin{equation}\label{eq:feid}
\gr{\Phi(-)} \cong F^{\KU},\qquad \gr{\Phi(-)} \wedge \bH\dQ \cong F^{\KU\wedge \bH\dQ},\qquad \gr{\Psi_M(-)}\cong F^{\bH M}. 
\end{equation}

\section{The Baum--Connes conjecture for étale groupoids}\label{sec:bc-conj}

In this section we discuss the Baum--Connes conjecture in the context of the equivariant $\KK$-category $\icKK^G$. This is a stable $\infty$-category whose homotopy category is the $G$-equivariant Kasparov category $\KK^G$ as defined by Le Gall~\cite{lg:kk}, as explained in Appendix~\ref{sec:kkg-stable}.
Throughout this section we assume that $G$ is a second countable, locally compact, Hausdorff, étale groupoid.

\subsection{Cellular approximation}\label{subsec:capp}

Let $H\subseteq G$ be an open subgroupoid. Denote by $\CstAlg^G$ the category of $G$-C$^*$-algebras and $G$-equivariant $\ast$-homomorphisms. We have natural induction and restriction functors, denoted $\Ind_H^G\colon \CstAlg^H\to \CstAlg^G$ and $\Res^G_H\colon \CstAlg^G\to \CstAlg^H$, inducing functors between the corresponding $\infty$-categories $\icKK^G$ and $\icKK^H$. The induction functor $\CstAlg^H\to \CstAlg^G$ can be concretely defined as $A\mapsto (C_0(G_{H^{0}})\,^{s}\!\otimes_{H^{0}} A)\rtimes H$. We are particularly interested in the case where $H=X$, i.e., the trivial subgroupoid with no nontrivial arrows and unit space equal to $X$ itself, in which case $\Ind_X^G$ is simply defined as $A\mapsto C_0(G)\,^{s}\!\otimes_X A$.

As proved in~\cite{val:kthpgrp}*{Section 2}, the functor $\Ind_H^G$ is left adjoint to $\Res^G_H$, when both are viewed as functors between the underlying triangulated categories $\KK^G$ and $\KK^H$.
We claim that this can be upgraded to adjointness of functors between $\infty$-categories in the sense of~\cite{lur:ht}*{Section 5.2}.

\begin{proposition}\label{prop:adj}
$\Ind_H^G$ is left adjoint to $\Res^G_H$, when both functors are viewed as functors between the stable $\infty$-categories $\icKK^G$ and $\icKK^H$.
\end{proposition}

\begin{proof}
By~\cite{lur:ht}*{Proposition 5.2.2.8}, it is enough to show that we obtain a homotopy equivalence from the natural transform
\begin{equation}\label{eq:transfadj}
\map_{\icKK^G}(\Ind_H^G(-),-)\to \map_{\icKK^H}(\Res^G_H\Ind_H^G(-),\Res^G_H(-))\to \map_{\icKK^H}(-,\Res^G_H(-))
\end{equation}
between functors $(\icKK^H)^{\mathrm{op}}\times \icKK^G\to \Spec$, where the last arrow is induced by the unit of the adjunction on the corresponding homotopy categories:
\begin{equation}\label{eq:unitadj}
B \sim_{} \Ind_H^H(B)= C_0(H)\otimes_{H^{(0)}} B\rtimes H \hookrightarrow C_0(G|_{H^{0}})\otimes_{H^{(0)}} B\rtimes H= \Res^G_H\Ind_H^G(B).
\end{equation}
Note that~\eqref{eq:unitadj} is given by the composition of a Morita equivalence and a $\ast$-homomorphism (both of these are $H$-equivariant). Thus the unit map is represented by a (proper) $H$-correspondence from $B$ to $\Res^G_H\Ind_H^G(B)$. Then Proposition~\ref{prop:simplicially-enriched-model-of-atiyah-transfer} will give us a morphism in $\icKK^H$ representing the unit of the adjunction at the level of $\infty$-categories. It is worth pointing out that when $H=X$, which is the case of interest for us,~\eqref{eq:unitadj} reduces to a $\ast$-homomorphism, hence we do not need correspondences to represent the unit of the adjunction in $\icKK^X$.

Since both sides in~\eqref{eq:transfadj} are compatible with fiber sequences it is sufficient to show that the transformation is an isomorphism after applying $\pi_0$. This is proved in~\cite{val:kthpgrp}*{Theorem 2.3}.
\end{proof}

Given a morphism $f\in\KK^G(A,B)$, we say it is \emph{weakly contractible} if $\Res^G_H(f)=0$. The weakly contractible morphisms form a homological ideal $\cI_H$, and the weakly contractible objects (i.e., those whose identity morphism is weakly contractible) form a localizing triangulated subcategory of $\KK^G$, denoted $\cN_H$. We say an object $P\in \KK^G$ is \emph{projective} if $\cI_H(P,A)=0$ for all objects $A\in \KK^G$.

In~\cite{val:kthpgrp} (and specifically for the case $H=X$, in~\cite{py:part-one}), building on previous work by Meyer and Nest~\cite{nestmeyer:loc}, it is proved that $\cN_H$ is part of a \emph{pair of complementary subcategories} $(\cP_H, \cN_H)$, where $\cP_H= \langle \ran\Ind_H^G \rangle $ is the localizing triangulated subcategory generated by the image of $\Ind_H^G$. 

In fact, the same result holds if we consider a countable family of open subgroupoids $(H_i)_{i\in I}$, we have the functors
\begin{align*}
 F &= (\Res^G_{H_i})_{i\in I}\colon \KK^G\to \prod_{i\in I}\KK^H,&
 F^\dagger &= \bigoplus_{i\in I} \Ind_{H_i}^G\colon \prod_{i\in I}\KK^H\to \KK^G
\end{align*}
and corresponding subcategories $\cN_I = \ker F$, $\cP_I = \ran \langle F^\dagger \rangle$.

Following the case of locally compact groups due to Meyer and Nest~\cite{nestmeyer:loc}*{Theorem 4.7}, and its elaboration based on the $\Ind$-$\Res$-adjunction~\cite{mey:two}*{Theorem 7.3}, we have the following.

\begin{theorem}[\cite{val:kthpgrp}]\label{thm:cpair}
The pair $(\cP_I,\cN_I)$ is a pair of complementary subcategories of $\KK^G$. The projective objects are the retracts of direct sums of objects in $\ran F^\dagger$. Moreover, for any $A\in \KK^G$ there is a functorial triangle, unique up to isomorphism, of the form
\begin{equation*}
P(A)\xrightarrow{D_A} A \to N(A) \to \Sigma P(A).
\end{equation*}
with $P(A)\in \cP_I$ and $N(A)\in \cN_I$.
\end{theorem}

Let us recall the (reduced) \emph{descent functor} $j_r^G\colon \CstAlg^G\to \CstAlg$ given by $A\mapsto A\rtimes_r G$ on objects. As usual it induces a functor between the respective $\icKK$-categories, $j^G_r\colon \icKK^G\to \icKK$. This functor was introduced at the level of $\KK^G$-theory in~\cite{lg:kk}*{Theorem 3.4}.

Let $\rK_*^{\mathrm{top}}(G; A)$ denote the domain of the assembly map $\mu^G_A$ for the \emph{Baum--Connes conjecture} on $(G,A)$.
In fact, $\rK_*(P(A) \rtimes_r G)$ is a model of this group, in the following sense.

\begin{theorem}[\cite{val:kthpgrp}*{Theorem 3.14}]\label{thm:bc}
There is a choice of countable subfamily $(H_i)_{i\in I}$ of the family of proper open subgroupoids of $G$ such that the corresponding triangle from Theorem~\ref{thm:cpair} satisfies:
\begin{itemize}
\item $\Res^G_H(D_A)$ is a $\KK^H$-equivalence for all proper open subgroupoids $H\subseteq G$.
\item There is a natural isomorphism $\phi^G_A\colon \rK_*^{\mathrm{top}}(G; A)\to \rK_*(P(A)\rtimes_r G)$ such that 
\[
\rK_*(j^G_r)(D_A)\circ \phi^A_G=\mu^A_G.
\]
\end{itemize}
\end{theorem}

\begin{remark}\label{rem:indresX}
When $G$ has torsion-free isotropy groups, we can take the singleton family $\{X\}$ containing the unit space of $G$.
We will mainly work with this choice in the next section.
\end{remark}

\begin{definition}
In the context of Theorem~\ref{thm:bc}, the morphism $D_A$ is called the \emph{Dirac morphism} and its domain $P(A)$ is called the \emph{cellular approximation}.
We say that $G$ \emph{satisfies the Baum--Connes conjecture} if the induced maps $\rK_i^{\mathrm{top}}(G; C_0(G^{(0)})) \to \rK_i(C^*_r G)$ for $A = C_0(G^{(0)})$ and $i = 0, 1$ are isomorphisms.
We say that $G$ \emph{satisfies the rational Baum--Connes conjecture} if the same holds after rationalization.
\end{definition}

\begin{remark}\label{rem:bc-amen}
The Baum--Connes conjecture holds for a large class of groupoids.
For example, based on work by Higson and Kasparov~\cite{higkas:bc}, Tu~\cites{tu:moy} has proved the conjecture for every locally compact, $\sigma$-compact, Hausdorff groupoid acting continuously and isometrically on a continuous field of affine Euclidean spaces (these groupoids are sometimes referred to as \emph{a-T-menable}, or with the \emph{Haagerup property}).
This class contains (topologically) amenable groupoids.
\end{remark}

The main advantage of working in $\icKK^G$ rather than $\KK^G$ is that we can give a more direct construction of the cellular approximation. To see this, we first introduce the notion of \emph{projective resolution}. Suppose $P_\bullet \to A$ is a chain complex where all the $P_i$'s are projective and the chain complex $F(P_\bullet)\to F(A)\to 0$ is split-exact. Then we say that $P_\bullet$ is a projective resolution of $A$ in $\KK^G$.

From Proposition~\ref{prop:adj} we also know that $F^\dagger$ is right adjoint to $F$.
By~\cite{MR3415698}, the composition $L=F^\dagger F$ is a comonad in the $\infty$-categorical sense, and $L^{\bullet+1} A \to A$ is an augmented simplicial object of $\icKK^G$ (in their language, the (strict) $2$-category $\underline{\smash{\mathrm{Adj}}}$ characterizing $\infty$-adjunction has $\Delta^\opo$ in the corner supporting this composition, hence we get an $\infty$-simplicial object by definition).

Let us set $\gr{L^{\bullet+1} A}=\colim_{\Delta^\opo} L^{\bullet+1} A \in \icKK^G$. We also refer to this colimit over the simplicial category as ``geometric realization'', see Section~\ref{subsec:georeal}.

\begin{theorem}\label{thm:cellapp-colim}
      For any separable C$^*$-algebra $A$ there is an equivalence $\gr{L^{\bullet+1} A}\cong P(A)$. Under this identification, the Dirac morphism $D_A$ corresponds to the natural map $\ell_A\colon \gr{L^{\bullet+1} A}\to A$.
\end{theorem}

\begin{proof}
Recall that split triangles in any triangulated category are isomorphic to direct sum triangles. This property is sufficient to construct a Dold--Kan correspondence in $\icKK^G$, as observed in~\cite{lur:ha}*{Section 1.2.4}. Thus, from the simplicial object $L^{\bullet+1} A$ we extract a chain complex $(C_\bullet,\delta_\bullet^{C_\bullet})$, where each $C_n$ is a direct summand of $L^{n+1}A$, in particular it is projective. To show that $C_\bullet\to A$ is a resolution, let $\epsilon$ be the unit of the adjunction in Proposition~\ref{prop:adj}, and consider $f_{-1}=\epsilon F$, $f_n=\epsilon FL^{n+1}$. Then $FL^{\bullet+1} A\to FA$ together with the $f_i$'s is a (left) contractible augmented simplicial object, so that its associated chain complex $FC_\bullet\to FA\to 0$ is exact. Let us set $P^\prime_n=\Sigma^n C_n$ and $\delta^{P^\prime_\bullet}_n=\Sigma^n \delta^{C_\bullet}_n$. Then $(P^\prime_\bullet,\delta^{P^\prime_\bullet}_\bullet)$ is an \emph{odd} projective resolution, as in the \emph{phantom tower} constructed in~\cite{mey:two}*{Section 3.1}. 

Now define $D_n$ to be the colimit of the $n$-skeleton of $L^{\bullet+1} A$. Note $D_0=P_0=C_0= LA$. By~\cite{lur:ha}*{Remark 1.2.4.3}, we have an exact triangle
\begin{equation*}
D_{n-1}\to D_{n}\xrightarrow{\sigma_n} P_{n}\xrightarrow{\kappa_{n}} \Sigma D_{n-1},
\end{equation*}
where the map $D_n\to D_{n+1}$ is induced by the skeletal inclusion, and $P^\prime_\bullet$ is canonically isomorphic to $P_\bullet$ as chain complexes. We also note we have maps $D_n\to A$ obtained by composing $\ell_A$ with the natural maps $\lambda_n\colon D_n\to \gr{L^{\bullet+1} A}$. Since $P(A)$ is defined through the phantom tower, we now need to relate the $D_n$'s to this construction. By~\cite{bbd:fais}*{Proposition 1.1.11}, we can obtain a diagram
\begin{center}
  
      \begin{tikzcd}
      D_{n-1} \arrow[d] \arrow[r] & D_{n} \arrow[d,"\ell_A\lambda_n"] \arrow[r] & P_{n} \arrow[d] \arrow[r] & \Sigma D_{n-1} \arrow[d] \\
      A \arrow[d] \arrow[r,"\id"]   & A \arrow[r] \arrow[d]   & 0    \arrow[d] \arrow[r] & \Sigma A \arrow[d] \\
      {N}_{n} \arrow[d] \arrow[r, dashed]  & {N}_{n+1} \arrow[r,dashed,"\varepsilon_n"] \arrow[d,"\gamma_{n+1}"]  & \Sigma P_n \arrow[d, "\id"] \arrow[r,dashed, "-\Sigma \pi_n"]  & \Sigma N_{n} \arrow[d,dashed,"\Sigma\gamma_n"]     \\
      \Sigma D_{n-1} \arrow[r]  & \Sigma D_{n} \arrow[r,"\Sigma \sigma_n"] & \Sigma P_{n}\arrow[r,"\Sigma\kappa_{n}"] & D_{n-1}               
      \end{tikzcd}
\end{center}
where every row and column is exact, and each square is commutative, except for the bottom right square which is anticommutative. We deduce the boundary map $P_{n+1}\to P_n$ coincides with $\varepsilon_n \pi_{n+1}$  and the objects $N_i$ form a phantom tower for $A$. Comparing the top row above with the triangle in~\cite{mey:two}*{equation~(3.3)} proves that the $D_n$'s form a \emph{cellular approximation tower} for $A$ (set $\tilde{A}_{n+1}=D_n$ in~\cite{mey:two}*{Definition 3.9}). Then by~\cite{mey:two}*{Proposition 3.18}, we have that the homotopy colimit in $\KK^G$ of the system $D_n\to D_{n+1}$ is isomorphic to $P(A)$. It is a standard fact that homotopy colimits in the homotopy category of a stable $\infty$-category are computed by $\infty$-colimits (see Proposition~\ref{prop:indinfcolim}). Therefore we have $\colim D_n \cong P(A)$, but we also have $\gr{L^{\bullet+1} A} \cong \colim D_n$ by construction (see~\cite{lur:ha}*{Remark 1.2.4.2}), giving us an equivalence $\gr{L^{\bullet+1} A}\cong P(A)$.

To identify $D_A$, note that the Dirac morphism is the map $\tilde{A}\to A$ in the vertical dashed triangle in~\cite{mey:two}*{diagram (3.7)}. This diagram is induced by the maps $\alpha_n\colon \tilde{A}_n\to A$ together with the codiagonal map $\nabla\colon \bigoplus A\to A$ in Meyer's notation. In our situation we have $\alpha_n=\ell_A\circ \lambda_n$, and $\nabla\bigoplus\alpha_n = \ell_A\circ\nabla\bigoplus\lambda_n$, from which we deduce that $\tilde{A}\to A$ is isomorphic to $\ell_A\colon \gr{L^{\bullet+1} A}\to A$.
\end{proof}

\begin{corollary}\label{cor:pgrdesc}
The map $j^G_r(D_A)$ is equivalent to the natural map $\ell_{A\rtimes_r G}\colon \gr{L^{\bullet+1} A\rtimes G}\to A\rtimes G$.
\end{corollary}

\begin{proof}
      This follows from Theorem~\ref{thm:cellapp-colim} once we know that $j^G_r$ preserves countable colimits. Since $j^G_r$ preserves semisplit exact sequences, it preserves zero and fiber sequences, thus it preserves all finite colimits~\cite{lur:ha}*{Prop.~1.1.4.1}. We also know $j^G_r$ preserves countable direct sums (coproducts), thus~\cite{lur:ht}*{Corollary 4.4.2.7} implies it preserves all countable colimits.
\end{proof}

We will next see that, when $G$ has torsion-free stabilizers, the domain of the map in Corollary~\ref{cor:pgrdesc} is closely related to the Crainic--Moerdijk homology.

\begin{remark}
When $G$ has torsion elements in stabilizers, we can still apply the corepresented functor $\map_{\icKK}(\dC, -)$ to $\ell_{A\rtimes_r G}$, providing a fully homotopy-theoretic interpretation of the Baum--Connes conjecture for étale groupoids, analogous to the Davis--Lück presentation of the assembly map~\cite{dl:bc}.
Of course, in this case the resulting homology groups are analogues of Bredon homology.
\end{remark}

\section{Generalized Chern character for ample groupoids}\label{sec:grchern}

In addition to the assumptions on $G$ in the previous section, let us also assume that it is ample.
Then $(G^{(n+1)})^\infty_{n=0}$ is a simplicial system of totally disconnected $G$-spaces, thus it gives a simplicial object of $\Span_{\pret}(\TDLC^G)$.
The associated quotient object $(G\backslash G^{(n+1)})^\infty_{n=0}$ is isomorphic to $(G^{(n)})^\infty_{n=0}$, a simplicial object in $\Span_{\pret}(\TDLC)$. 

We have seen in~\eqref{eq:fe} that for any topological spectrum $\bE$ we obtain a simplicial object $F^\bE(G^{(\bullet)})$  in $\TopSpec$. It will be convenient to write $G^\bullet \bE$ in place $F^\bE(G^{(\bullet)})$ to avoid clutter.
For a simplicial  spectrum $\bA^\bullet$ we will write $\gr{\bA^\bullet}$ to denote its realization, i.e., the $\infty$-colimit $\colim_{\Delta^{\text{op}}}\bA^\bullet$. 

Thanks to the work carried out in Section~\ref{sec:span-tot-disc} and~\ref{sec:k-theory-transfer}, now we can define the Chern character of $G$ at the level of simplicial spectra. 

\begin{definition}\label{def:mainch}
The \emph{Chern character} for the ample groupoid $G$ is defined as the map 
\[
G^{\bullet}\KU\to \bigvee_{n\in\dZ} \Sigma^{2n} \bH C_c(G^{(\bullet)},\dQ)
\]
between simplicial objects in $\TopSpec$, obtained as the following composition
\[
G^{\bullet}\KU\to G^{\bullet}\KU\wedge \bH\dQ \xrightarrow{\cong} G^{\bullet}(\KU\wedge \bH\dQ)\xrightarrow{\cong} G^{\bullet}(\bH\dQ[u,u^{-1}])\xrightarrow{\cong} \bigvee_{n\in\dZ} \Sigma^{2n} \bH C_c(G^{(\bullet)},\dQ),
\]
where the first map is the standard inclusion $X\to X\wedge Y$, and the following equivalences are due to~\eqref{eq:feid}, the Chern--Dold equivalence in~\eqref{eq:chern-equiv}, and Proposition~\ref{prop:final-transfer}, in order from left to right.
\end{definition}

Definition~\ref{def:mainch} is the Chern character of $G$ ``defined at the level of spectra'' that was promised in Theorem~\ref{thm:a-def-ch}.
Notice that we do not need $G$ to be torsion-free to make this definition.

\begin{proposition}\label{prop:s-hom}
We have the following isomorphism of abelian groups:
\[
\pi_*(\gr{G^{\bullet}\KU})\otimes \dQ \cong \pi_*(\gr{G^{\bullet}\KU}\wedge \bH\dQ)\cong \dH_*(G; \rK_\bullet(\dC)\otimes \dQ).
\]
\end{proposition}

\begin{proof}
Since the smash product preserves colimits, from Definition~\ref{def:mainch} we have 
\begin{equation}\label{eq:hgch}
\gr{G^{\bullet}\KU}\wedge \bH\dQ \cong \gr{G^{\bullet}\KU\wedge \bH\dQ}\cong \bigvee_{n\in\dZ} \Sigma^{2n}\gr{ \bH C_c(G^{(\bullet)},\dQ)}.
\end{equation}
Now the claim follows from $\pi_*\gr{\bH C_c(G^{(\bullet)},\dQ)}\cong \rH_*(C_c(G^{(\bullet)},\dQ))$.
Conceptually, this follows from the fact that $H$ implements an equivalence (of stable $\infty$-categories) between rational chain complexes and $\bH\dQ$-module spectra, which interchanges homotopy groups and homology groups, and is such that the realization of a simplicial $\bH\dQ$-module is modeled by the direct sum total complex construction. 

A more elementary approach is to consider the rightmost realization in~\eqref{eq:hgch}, and note that the standard spectral sequence associated to the homotopy colimit converges (see~\cite{lur:ha}*{Variant 1.2.4.9 and Proposition 1.4.3.6}). The associated $E_1$-page reads $\pi_q(\bH C_c(G^{(p)},\dQ))$. From the properties of Eilenberg--MacLane spectra, this page is concentrated in the $0$-th row, so the homotopy groups of $\gr{\bH C_c(G^{\bullet},\dQ)}$ are isomorphic to the homology groups of the chain complex $\pi_0(\bH C_c(G^{\bullet},\dQ))\cong C_c(G^{\bullet},\dQ)$.
The boundary maps are identified in Proposition~\ref{prop:final-transfer}, thus comparing with the description in~\eqref{eq:ghcc} gives the claim.
\end{proof}

Following up on Remark~\ref{rem:indresX}, we now consider the simplicial object $L^{\bullet+1} A$ obtained by choosing $F=\Res^G_X$ and $F^\dagger=\Ind_X^G$ (recall $X=G^{(0)}$ is the unit space of $G$).

\begin{proposition}\label{prop:id-simp-obj}
There is an equivalence of simplicial objects between $\map_{\icKK}(\dC,L^{\bullet+1} C_0(X)\rtimes G)$ and $G^{\bullet} \KU$ (after levelwise geometric realization). In particular, taking the colimit over $\Delta^{\text{op}}$, we have an equivalence
\[
\gr{\map_{\icKK}(\dC,L^{\bullet+1} C_0(X)\rtimes G)} \cong \gr{G^{\bullet}\KU}.
\]
\end{proposition}
\begin{proof}
There is a $G$-equivariant $\ast$-isomorphism $L^{\bullet+1} C_0(X)\cong C_0(G^{(\bullet+1)})$. The counit of the adjunction $\Ind_X^G\dashv \Res^G_X$ gives the augmentation map $\epsilon_X\colon C_0(G)\to C_0(X)$. All the face maps in $C_0(G^{(\bullet+1)})$ are obtained from $\epsilon$, for example the two maps $C_0(G^{(2)})\to C_0(G)$ are given by $\epsilon_{C_0(G)}$ and $L(\epsilon_X)$. 

Following the definition of the counit given in~\cite{val:kthpgrp}*{Theorem 2.3}, the augmentation map $\epsilon_X$ is described as
\[
C_0(G)\hookrightarrow C_0(G)\rtimes G \sim C_0(X),
\]
i.e., it is given by a $\ast$-homomorphism (inclusion) followed by a Morita equivalence.
Comparing with the discussion below Lemma~\ref{lem:grcorr-comp}, we see that $\epsilon_X$ is represented by the (equivariant) correspondence  $E^G_s=E(\Omega^G_s)$, where $\Omega^G_s$ is the groupoid correspondence associated to the local homeomorphism $s\colon G\to X$.

To be precise, we have that $T_\icKK^G(E^G_s)$ (using Proposition~\ref{prop:simplicially-enriched-model-of-atiyah-transfer}) is a lift of $\epsilon_X\in \KK^G(C_0(G),C_0(X))$ to $\icKK^G$. Then it follows that all the other face maps are represented by $E^G_{\partial^i_n}$, where $\partial^i_n$ is the $G$-equivariant $i$-th face map $G^{(n+1)}\to G^{(n)}$, $i=0,\dots,n$. By applying Proposition~\ref{prop:span-corr-comm}, we see that $L^{\bullet+1} C_0(X)\rtimes G$ is isomorphic (in $\icKK$) to $C_0(G^{(\bullet)})$ with face maps $T_\icKK(E_{d^i_n})$, where $d^i_n\colon G^{(n)}\to G^{(n-1)}$ is the standard face map, $i=0,\dots, n$. Again, to be precise, there is a natural isomorphism $T_\icKK(E^G_{\partial^i_n}\rtimes G)\cong  j_r^G T_\icKK^G(E^G_{\partial^i_n})$, and we are applying Proposition~\ref{prop:span-corr-comm} to $E^G_{\partial^i_n}\rtimes G$ to deduce the isomorphism of simplicial objects 
\[
L^{\bullet+1} C_0(X)\rtimes G\cong T_\icKK(E(G\backslash G^{(\bullet+1)}))\cong T_\icKK(E(G^{(\bullet)})).
\]
The claim follows from~\eqref{eq:feid} because by definition $\Phi(G^{(\bullet)})= \map_{\icKK}(\dC,T_\icKK(E(G^{(\bullet)})))$.
\end{proof}

\begin{proposition}
Suppose $G$ is a second-countable, locally compact, Hausdorff, ample groupoid. Assume $G$ is torsion-free. The following isomorphism holds:
\[
\rK_*^{\mathrm{top}}(G; C_0(X))\cong  \pi_*(\gr{G^{\bullet}\KU}).
\]
\end{proposition}
\begin{proof}
Recall from Theorem~\ref{thm:cellapp-colim} that $P(C_0(X))\rtimes G$ is isomorphic to $\colim_{\Delta^{\mathrm{op}}} {L^{\bullet+1} C_0(X)\rtimes G}$ in $\icKK$.
Since $\map_{\icKK}(\dC,-)$ commutes with colimits, we have
\[
\map_{\icKK}(\dC,P(C_0(X))\rtimes G)\cong \colim_{\Delta^{\mathrm{op}}}\map_{\icKK}(\dC,L^{\bullet+1} C_0(X)\rtimes G).
\]
Now Theorem~\ref{thm:bc} and Proposition~\ref{prop:id-simp-obj} imply the claim.
\end{proof}

Combining the isomorphisms from the above propositions, we now obtain Theorem~\ref{thm:a-def-ch}.

\begin{theorem}\label{thm:ch-ab}
Suppose $G$ is a second-countable, locally compact, Hausdorff, ample groupoid. Assume $G$ is torsion-free.
The Chern character of $G$ from Definition~\ref{def:mainch} induces the following morphism of abelian groups, which is an isomorphism after rationalizaton,
\begin{equation*}
\rK_i^{\mathrm{top}}(G; C_0(X)) \to \bigoplus_{k\in \dN} \rH_{i+2k}(G; \dQ) \quad (i = 0, 1).
\end{equation*}
\end{theorem}

The rational Baum--Connes conjecture allows us to replace the domain with the K-groups of $C^*_r G$, and we obtain Theorem~\ref{thm:main-thm}.

\begin{corollary}\label{cor:hk-iso}
If in addition $G$ satisfies the rational Baum--Connes conjecture, then the rational HK conjecture for $G$ holds, i.e., we have the following isomorphism of abelian groups:
\begin{equation*}
\rK_i(C^*_r G)\otimes \dQ \cong \bigoplus_{k\in \dN} \rH_{i+2k}(G; \dQ) \quad (i = 0, 1).
\end{equation*}
\end{corollary}

\section{Applications and further discussion}\label{sec:appl}

In this section we present other applications of our main results and discuss related settings on homology and $\rK$-theory of groupoids that are closely related but not exactly covered by this paper.

\subsection{Integral HK conjecture}

Theorem~\ref{thm:main-thm} proves a modified form of the original HK conjecture formulated by Matui in~\cite{matui:hkpub}*{Conjecture 2.6} (and implicitly in his earlier work~\cite{matui:hk}*{Section 3}). The formulation in Theorem~\ref{thm:main-thm} appeared in print in~\cite{bdgw:matui}*{Remark 5.12}, however it has been known to experts in the field since shortly after the appearance of~\cite{matui:hkpub}. 

Matui's original formulation asks if the following isomorphism holds:
\[
\rK_i(C^*_r G) \cong \bigoplus_{k \in \dN} \rH_{i + 2 k}(G; \dZ) \qquad (i = 0, 1),
\]
without specifying a particular implementing morphism. This conjecture is different from Theorem~\ref{thm:main-thm} in a number of ways, which we now comment on.

First of all, the original conjecture prescribed an \emph{integral isomorphism} (i.e., an isomorphism \emph{before} rationalization).
This may be explained in view of the well-known fact that the Chern character of totally disconnected spaces is an integral isomorphism (note that groupoid homology specializes to negatively-graded compactly supported cohomology when calculated on a space, viewed as a trivial groupoid, see~\cite{cramo:hom}*{Section 3.5}).
However, Deeley has shown that the integral version of Theorem~\ref{thm:main-thm} does not hold, by providing a counterexample groupoid $G$ that is also \emph{principal}, i.e., an equivalence relation as a groupoid~\cite{deeley:counthk}.

Matui's original formulation also imposed a topological condition on isotropy, namely that the interior of the isotropy subgroupoid of $G$ equals $G^{(0)}$; a groupoid of this form is called \emph{essentially principal}. Scarparo has produced a counterexample to the original HK conjecture which is essentially principal, but not torsion-free~\cite{scarparo:homod}.

Overall we can say that the failure of Matui's HK conjecture is tied to torsion phenomena in $\rK$-groups, or in the isotropy groups of the groupoid itself.

It should be noted that Matui's original motivation, that is the study of topological dynamical systems and their invariants, is one where virtually all groupoids involved happen to be amenable, in particular they satisfy the Baum--Connes conjecture (see Remark~\ref{rem:bc-amen}).
This may account for why this condition was overlooked in Matui's original formulation.

The context of topological dynamical systems is also likely the explanation why the original conjecture imposed \emph{minimality} on $G$, i.e., that $G$-orbits be dense, as interesting dynamical systems often exhibit this behavior. Minimality is not relevant for Theorem~\ref{thm:main-thm}.

\subsection{HK conjecture with coefficients and groupoid models for classifiable C$^*$-algebras}\label{sec:hk-coeff}

Suppose $G$ satisfies the assumptions of Theorem~\ref{thm:main-thm}, and let $B$ be a $G$-C$^*$-algebra.
In~\cite{py:part-one}, we have proved that the $\rK$-groups of $B$ are $\dZ G$-modules, and that they appear in a convergent spectral sequence as follows:
\[
E^2_{pq}=\rH_p(G; \rK_q(B))\Rightarrow \rK_{p+q}(B\rtimes_r G).
\]
In our forthcoming work~\cite{py:upcoming}, we prove the corresponding generalization of Theorem~\ref{thm:main-thm}:
\begin{equation}\label{eq:hk-coeff}
\rK_i(B\rtimes_r G)\otimes \dQ \cong \dH_i(G; \rK_\bullet(B)\otimes \dQ).
\end{equation}

The case where $B$ is the C$^*$-algebra of compact operators is of particular interest for the following reason. Recall that a \emph{twist} $\Sigma$ over a topological groupoid $G$ is a central groupoid extension
\[
G^{(0)} \times \dT \longrightarrow \Sigma \longrightarrow G.
\]
To this data one can associate a \emph{twisted} groupoid C$^*$-algebra denoted $C^*_r(G,\Sigma)$.

If $A$ is a separable C$^*$-algebra with a Cartan subalgebra, then it is isomorphic to $C^*_r(G,\Sigma)$ for a second countable, locally compact, Hausdorff, étale groupoid $G$ with twist $\Sigma$~\cite{ren:cartan}*{Theorem 5.6}. In addition, $G$ is principal (thus torsion-free) if and only if the Cartan subalgebra has the unique extension property. We also know that $G$ is amenable (thus it satisfies the Baum--Connes conjecture) if and only if $A$ is nuclear~\cite{tak:fell}*{Theorem 5.4}. A generalization of the Packer--Raeburn stabilization trick~\cite{willerp:parae} shows that there exists a $C_0(G^{(0)})$-module $E$ and a groupoid dynamical system $(\cK(E), G, \alpha)$ such that $\cK(E)\rtimes_{r,\alpha} G$ is Morita equivalent to $C^*_r(G,\Sigma)$. Choosing the C$^*$-algebra of compact operators in place of $B$ in~\eqref{eq:hk-coeff}, this leads to the computation of $\rK_i(C^*_r(G,\Sigma)) \otimes \dQ$ in terms of hyperhomology as above.

\smallskip

Let us discuss how to compute $K_*(A)\otimes \dQ$ from the homology of $G$ when $A$ is classifiable. Recall $A$ is either stably finite or purely infinite. In the former case, $A$ admits a trivially twisted groupoid model if $\rK_0(A)$ is torsion-free~\cite{li:diag}*{Corollary 1.8}; in the latter case, $A$ is a Kirchberg algebra and we can choose a graph-based model which in particular is ample with trivial twist~\cites{cfh:kirchample,spiel:kirchgraph}.
This groupoid model is essentially principal, but we cannot exclude torsion, so Theorem~\ref{thm:main-thm} applies only to Kirchberg algebras (with UCT) admitting a torsion-free model. 

When $A$ is stably finite the groupoid model can always be chosen principal~\cite{li:diag}*{Theorem 1.2}, so we can apply Theorem~\ref{thm:main-thm} when $\rK_0(A)$ is torsion-free and the model is ample. The condition that $G$ is ample excludes projectionless C$^*$-algebras, such as the Jiang--Su algebra.

\subsection{Smale spaces and Putnam's homology}

Inspired by the theory of basic sets of Axiom A diffeomorphisms, Ruelle has introduced a class of topological hyperbolic dynamical systems called \emph{Smale spaces}~\cite{ruelle:thermo}. Based on ideas by Bowen~\cite{bow:mA}, Putnam has introduced a homology theory for Smale spaces~\cite{put:HoSmale}. He has also studied groupoids and C$^*$-algebras associated to these systems in~\cites{put:algSmale,put:funct}.

For a non-wandering Smale space $(X,f)$ with totally disconnected stable sets, its unstable groupoid $R^u(X,f)$ can be chosen to be an ample groupoid satisfying the hypotheses of Theorem~\ref{thm:main-thm}. Let us denote Putnam's stable homology groups by $\rH^s_*(X,f)$. Let $U(X,f)=C^*(R^u(X,f))$ denote the unstable C$^*$-algebra associated to the Smale space $(X,f)$. 

\begin{corollary}\label{cor:hk-iso-smale}
Let $(X,f)$ be a non-wandering Smale space with totally disconnected stable sets. The following isomorphism holds:
\[
\rK_i(U(X,f))\otimes \dQ \cong \bigoplus_{k\in \dN} \rH^s_{i+2k}(X,f)\otimes \dQ \quad (i = 0, 1).
\]
\end{corollary}

\begin{proof}
We apply Theorem~\ref{thm:main-thm} to $G=R^u(X,f)$. Notice that $G$ is an equivalence relation, thus it has trivial isotropy groups; it is also amenable~\cite{put:spiel}*{Theorem 1.1}, hence it satisfies the Baum--Connes conjecture~\cite{tu:moy}. We have proved that $\rH^s_{*}(X,f)$ is isomorphic to $\rH_*(R^u(X,f)|_T,\dZ)$ for any choice of transversal $T$~\cite{py:part-two}*{Theorem 5.1}.
\end{proof}

We have proved in~\cite{py:part-three}*{Theorem 5.1} that $U(X,f)$ has finite rank $\rK$-groups for any non-wandering Smale space $(X,f)$.
When $(X,f)$ has totally disconnected stable sets, we can use Corollary~\ref{cor:hk-iso-smale} to give a more precise estimate on the rank.
Note that by construction $\rH^s_k(X,f)$ has finite rank and it vanishes for $k$ large enough~\cite{put:HoSmale}*{5.1.12}.

\begin{corollary}
Let $(X,f)$ be a non-wandering Smale space with totally disconnected stable sets. The following equality holds:
\[
\rank \rK_i(U(X,f)) =  \sum_{k\in \dS} \rank \rH^s_{i+2k}(X,f) \quad (i = 0, 1).
\]
\end{corollary}

\subsection{Topological full groups and algebraic \texorpdfstring{$\rK$}{K}-theory}

In~\cite{li:algk}, Li has proved a remarkable result connecting the homology of an ample groupoid with the homology of its associated topological full group.
A key ingredient in Li's strategy is the construction of the \emph{algebraic $\rK$-theory spectrum} associated to an ample groupoid $G$. Let us denote this object $\bK^{\mathrm{alg}}(G)$.
It is defined as the Segal $\rK$-theory of the symmetric monoidal category generated by the \emph{embedding category} $\Emb_c(G)$ from~\cite{moe:weakhomtype} whose objects are compact open sets of $G^{(0)}$ by using its \emph{assembler} structure; the details of this construction are not important for our purposes, the reader may consult~\cites{li:algk,klmms:scissors,zach:assk,segal:catcoh} for more information.

Recall that the Karoubi conjecture, as proved by Suslin and Wodzicki~\cite{suswod:karoubi}, states that the topological $\rK$-groups of a C$^*$-algebra $A$ are isomorphic to the algebraic $\rK$-groups of the stabilized C$^*$-algebra $A \otimes \cK$ (viewed as a discrete ring).
The following result shows that the algebraic $\rK$-theory spectrum of $G$ serves as a substitute of $\BGL(A \otimes \cK)^+$ after rationalization.
It can also be interpreted as determining the rational homotopy type of $\bK^{\mathrm{alg}}(G)[u,u^{-1}]$.

\begin{corollary}\label{cor:ktop-kalg}
Suppose $G$ is a second-countable, locally compact, Hausdorff, ample groupoid. There is an isomorphism
      \[
      \pi_*(\gr{G^\bullet \KU})\otimes \dQ \cong  \bigoplus_{k\in \dZ}\pi_{*+2k} (\bK^{\mathrm{alg}}(G))\otimes \dQ.
      \]
In particular, if $G$ is torsion-free and satisfies the rational Baum--Connes conjecture, 
\[
\rK_*(C^*_r G) \otimes \dQ \cong  \bigoplus_{k\in \dZ}\pi_{*+2k} (\bK^{\mathrm{alg}}(G))\otimes \dQ
\]
\end{corollary}

\begin{proof}
On the one hand, we have
\[
\pi_{n} (\bK^{\mathrm{alg}}(G))\otimes \dQ \cong \pi_n(\bK^{\mathrm{alg}}(G) \wedge \bH \dQ),
\]
see~\eqref{eq:homotop-grp-rationalization}.
By the Hurewicz isomorphism~\cite{schede-book-draft-symm-spec}*{Proposition II.6.30}, $\pi_n(\bK^{\mathrm{alg}}(G) \wedge \bH \dQ)$ is isomorphic to the (reduced) homology of $\bK^{\mathrm{alg}}(G)$ with rational coefficients.
By~\cite{li:algk}*{Theorem 4.18}, this is isomorphic to $\rH_n(G; \dQ)$.
On the other hand, we saw that $\pi_m(\gr{G^\bullet \KU})\otimes \dQ$ agrees with $\bigoplus_k \rH_{m + 2k}(G; \dQ)$ in Proposition~\ref{prop:s-hom}.
Under the rational Baum--Connes conjecture, we can use $\rK_m(C^*_r G)$ instead of $\pi_m(\gr{G^\bullet \KU})$ by Theorem~\ref{thm:main-thm}. 
\end{proof}

\begin{remark}
The \emph{integral} computation of the homotopy groups of $\bK^{\mathrm{alg}}(G)$ remains an interesting problem, especially in view of recently established connections with Zakharevich's higher scissors congruence $\rK$-theory~\cites{klmms:scissors,zach:assk}.
\end{remark}

Regarding the topological full group, we have the following consequence. Let $\mathcal{R}$ denote the full equivalence relation on the natural numbers. For an ample groupoid $G$, we denote by $F(G)$ its topological full group. When $G^{(0)}$ is compact, the topological full group $F(G)$ is the group of compact open bisections with source and range equal to $G^{(0)}$. Multiplication is defined by multiplication of
bisections. If the unit space of $G$ is not compact, we define $F(G)$ as the colimit $F(G)=\varinjlim F(G|_U)$ taken over all compact open subspaces $U\subseteq G^{(0)}$ ordered by inclusion (when $U\subseteq V$, the bonding map $F(G|_U)\to F(G|_V)$ sends a bisection $\sigma$ to $\sigma \sqcup (V\smallsetminus U)$).

Given a $\dZ/2\dZ$-graded rational vector space $E_\bullet$, we consider the universal $\dZ/2\dZ$-graded unital commutative algebra $\uCom(E_\bullet)$.
Namely, it is a unital $\dQ$-algebra generated by the vectors of $E_\bullet$ subject to the relation
\[
x y = (-1)^{\absv{x} \absv{y}} y x
\]
for homogeneous elements $x$, $y$.
We treat it as a graded algebra by the grading
\[
\absv{x_1 \cdots x_n} = \sum_{i = 1}^n \absv{x_i}
\]
for homogeneous elements $(x_i)_i$.

\begin{corollary}\label{cor:fg-ext-sym}
Suppose $G$ satisfies the hypotheses of Corollary~\ref{cor:ktop-kalg}. There is an isomorphism
\[
\dH_*(F(G\times\mathcal{R}), \rK_\bullet(\dC)\otimes \dQ) \cong \uCom(\pi_\bullet(\gr{G^\bullet\KU}) \otimes \dQ)
\]
of $\dZ/2\dZ$-graded vector spaces.
In particular, if $G$ is torsion-free and satisfies the rational Baum--Connes conjecture, there is an isomorphism
\begin{equation}\label{eq:ext-sym-bc}
\dH_*(F(G\times\mathcal{R}), \rK_\bullet(\dC)\otimes \dQ)  \cong \uCom(\rK_\bullet(C^*_r G)\otimes \dQ)
\end{equation}
of $\dZ/2\dZ$-graded vector spaces. 
\end{corollary}

\begin{proof}
By~\cite{li:algk}*{Corollary 6.1}, we have
\[
\dH_*(F(G\times\mathcal{R}), \rK_\bullet(\dC)\otimes \dQ) \cong \uCom(H_\bullet(G; \dQ)),
\]
where we treat $\rH_\bullet(G; \dQ)$ as a $\dZ/2\dZ$-graded vector space in the obvious way.
Proposition~\ref{prop:s-hom} and Theorem~\ref{thm:main-thm} say that $\rH_\bullet(G; \dQ)$ is isomorphic to the graded vector spaces $\pi_\bullet(\gr{G^\bullet\KU}) \otimes \dQ$ and $\rK_\bullet(C^*_r G)\otimes \dQ$ under the respective assumptions..
\end{proof}

\begin{remark}\label{rem:fgr-fg}
In Corollary~\ref{cor:fg-ext-sym} we can replace $F(G\times \cR)$ with $F(G)$ when $G^{(0)}$ has no isolated points, and $G$ is minimal and has comparison, see~\cite{li:algk}*{Corollary 6.1} for more details.
In analogy with Corollary~\ref{cor:hk-iso}, if we drop the assumption on the Baum--Connes conjecture, then~\eqref{eq:ext-sym-bc} holds by plugging in the left-hand side of the assembly map, i.e., replacing $\rK_*(C^*_r G)$ with $\rK_*^{\mathrm{top}}(G; C_0(G^{(0)}))$.
\end{remark}

\begin{remark}
To illustrate the utility of Corollary~\ref{cor:fg-ext-sym}, we make the following observation, already contained in~\cite{li:algk}*{Introduction}.
As observed in~\cite{nekra:v}, Thompson's group $V$ is isomorphic to the topological full group of the ample groupoid $G_2$, known as the Renault--Deaconu groupoid for the full shift on two symbols.
It is well-known that the C$^*$-algebra $C^*_r G_2$ is (stably) isomorphic to the Cuntz algebra $\cO_2$~\cite{cuntz:otwo}.
Since $\rK_*(\cO_2)$ is trivial, $V$ can be seen to be rationally acyclic by Remark~\ref{rem:fgr-fg}, which is a result by Brown~\cite{brown:fp}.
In fact, $V$ is even integrally acyclic~\cite{sw:ht}.
It is known that the HK isomorphism of Corollary~\ref{cor:hk-iso} holds integrally for $G_2$ (the spectral sequence in~\cite{py:part-one} collapses, see also~\cite{fkps:hk}), thus all homology groups $\rH_*(G_2,\dZ)$ vanish, and in turn this implies the vanishing of $\tilde{\rH}_*(V,\dZ)$ by Li's result in~\cite{li:algk}*{Section 6.2}.
\end{remark}

\subsection{Beyond the ample case}

In our forthcoming work~\cite{py:upcoming}, we will prove a generalization of Theorem~\ref{thm:a-def-ch}, namely, give a rational isomorphism
\[
\rK_i^{\mathrm{top}}(G; C_0(G^{(0)})) \to \bigoplus_{k\in \dZ} \rH_{i+2k}(G; \dQ) \quad (i = 0, 1)
\]
for second-countable, locally compact, Hausdorff étale groupoids $G$.
For a locally compact Hausdorff space regarded as a groupoid, this reduces to the classical Chern character isomorphism of~\eqref{eq:chdoldconc}, since groupoid homology for a trivial groupoid computes compactly supported cohomology with grading inverted, see~\cite{cramo:hom}*{Section 3.5}.

More generally, when $G$ is the transformation groupoid $\Gamma \ltimes X$ for an action of a (countable) discrete group $\Gamma$ on a locally compact, second countable, Hausdorff space $X$, such an isomorphism follows from a recent paper by Deeley and Willett~\cite{dw:hk}. 

\subsection{Groups, group actions, and other variants}

In case of transformation groupoids arising from group actions on spaces, the Baum--Connes conjecture is well understood.
In particular, we find isomorphisms analogous to Theorem~\ref{thm:main-thm} in the literature.
For example, when $G$ is a discrete group (possibly with torsion elements) regarded as a groupoid, a rational isomorphism between the operator K-groups and the Bredon group homology groups follow from Lück's work~\cite{lueck:chern}.

Turning to the case of transformation groupoids, Raven~\cite{raven:phd} established a rational isomorphism between the equivariant $\KK$-groups of (proper, finite) $G$-CW-complexes and the bivariant homology theory introduced by Baum and Schneider~\cites{bs:chern}. 
This applies to actions of totally disconnected groups as well. In particular, using equivariant cylic homology, Voigt shows that Baum--Schneider homology is rationally isomorphic to the equivariant Kasparov group for algebras of finite-dimensional, locally finite (proper) $G$-simplicial complexes~\cite{voigt:chern}.

It is an interesting task to explore this direction for more general topological groupoids. From the standpoint of noncommutative geometry, it is natural to seek for a construction of the map above using the Chern--Connes character and the periodic cyclic homology of a smooth subalgebra of~$C^*_r G$. 
Recent work of Pagliuca and Voigt provides a promising first step in this direction~\cite{pv:cycg}.

In the case of torsion-free groups acting on spaces, Raven's Chern character map can be refined to have $\rK_*^{\mathrm{top}}(G; C_0(X))$ as the domain and the groupoid homology $\rH_*(G;\dC)$ as the codomain (see~\cite{dw:hk}).
In particular, the \emph{étale} rational HK conjecture for transformation groupoids with torsion-free stabilizers that satisfy the Baum--Connes conjecture follows from this.
In the presence of torsion in the stabilizers, one should consider the ``blow-up'' groupoid associated to $G$~\cite{bc:chern}, see~\cite{dw:hk}*{Theorem 3.1}.

\smallskip

The previous two variants involve modifying the right-hand side, i.e., the homology side, of Theorem~\ref{thm:main-thm}. We would like to propose a third variant where the left-hand side, i.e., the $\rK$-theory side, is changed. This approach implicitly appeared in our previous work~\cite{py:part-one}*{Section 5.3}. Let $G$ be an étale groupoid with unit space $X$. Let $P^\prime(C_0(X))$ denote the cellular approximation associated to the family of subgroupoids $\{X\}$, as explained in Section~\ref{subsec:capp}. Note that in general we do not have $P(C_0(X))\cong P^\prime(C_0(X))$ unless $G$ is torsion-free (see Remark~\ref{rem:indresX}). From Theorem~\ref{thm:cellapp-colim}, Proposition~\ref{prop:s-hom}, and Proposition~\ref{prop:id-simp-obj}, we know that $\rK_*(P^\prime(C_0(X))\rtimes G)\cong  \pi_*(\gr{G^{\bullet}\KU})$. On this basis we formulate the following variant of the HK conjecture:
\begin{equation}\label{eq:hk-pprime}
\rK_*(P^\prime(C_0(X))\rtimes G)\otimes \dQ \cong \dH_*(G;\rK_\bullet(\dC)\otimes\dQ).
\end{equation}
In some cases an explicit model of $P^\prime(C_0(X))$ is available. This happens for example for the dihedral group action featuring in Scarparo's counterexample~\cite{scarparo:homod}, see~\cite{py:part-one}*{Corollary 5.5}. Following Section~\ref{sec:hk-coeff}, a version of~\eqref{eq:hk-pprime} with coefficients can be obtained by replacing $P^\prime(C_0(X))$ with $P^\prime(A)\cong P^\prime(C_0(X))\otimes_X A$ and $\rK_\bullet(\dC)$ with $\rK_\bullet(A)$ for a separable $G$-C$^*$-algebra $A$.

\appendix
\section{Categorical structures behind equivariant \texorpdfstring{$\KK$}{KK}-theory}\label{sec:kkg-stable}

\subsection{Stable $\infty$-category}

Let $G$ be a second countable, locally compact, and Hausdorff groupoid. We will assume $G$ admits a Haar system (see~\cite{ren:group}*{Chapter 1, Section 2}).

Let $\CstAlg_s^G$ denote the category of \emph{separable} $G$-C$^*$-algebras and equivariant $\ast$-homomorphisms.
The $G$-equivariant Kasparov groups $\KK^G(A,B)$ for objects in this category~\cite{lg:kk} define an additive category $\KK^G$.
Let $\kk^G\colon \CstAlg_s^G \to \KK^G$ be the natural functor.
This functor $\kk^G$ is universal among homotopy invariant, stable, and split-exact functors into additive categories, see~\cite{py:part-one}.
Moreover, the category $\KK^G$ is triangulated, and its exact triangles are the mapping cone triangles~\cites{val:kthpgrp,py:part-one}.

Let us show that there is a stable $\infty$-category that realizes $\KK^G$ as its homotopy category, following~\citelist{\cite{landnik:comp}\cite{bel:kkstable}}.

\begin{definition}
Let $\cW\subset \CstAlg_s^G$ be the class of $*$-homomorphisms $\phi$ such that $\kk^G(\phi)$ is an isomorphism (i.e., $\KK^G$-equivalences).
We define the $(\infty, 1)$-category $\icKK^G$ as the Dwyer--Kan localization $\CstAlg_s^G[\cW^{-1}]$, together with the canonical $\infty$-functor $\ickk^G \colon \CstAlg_s^G \to \icKK^G$.
\end{definition}

By the universality of localization we have an $\infty$-functor $\icKK^G \to \KK^G$ which is identity on objects.
Since the target category is an ordinary category, this induces an ordinary functor from the homotopy category, which we denote by $\ho_0\colon \ho(\icKK^G)\to \KK^G$.
Our main goal is the following.

\begin{theorem}\label{thm:kkg-infty-stable-tri}
The $\infty$-category $\icKK^G$ is stable, and $\ho_0$ is an equivalence of triangulated categories.
\end{theorem}

Let us call a morphism $f\in \CstAlg^G$ a \emph{semisplit surjection} if it is surjective and it admits a $G$-equivariant completely positive contractive section.

\begin{definition}[\cite{MR341469}]
A \emph{category of fibrant objects} is given by a category $\cC$ with a terminal object $*$, together with two classes of morphisms $\cW$ (``weak equivalences'') and $\cF$ (``fibrations'') such that:
\begin{itemize}
\item both classes contain the isomorphisms in $\cC$, and are closed under compositions;
\item the unique morphism $X \to *$ belongs to $\cF$ for all objects $X$;
\item for any $f \colon X \to Z$ in $\cF$ and for any morphism $g \colon Y \to Z$, a pullback
\begin{equation}\label{eq:pullback}
\begin{tikzcd}
  V \arrow[r] \arrow[d, "f'"] & X \arrow[d, "f"] \\
  Y \arrow[r, "g"] & Z
\end{tikzcd}
\end{equation}
exists in $\cC$, and $f'$ belongs to $\cF$;
\item if furthermore $f$ belongs to $\cW$, then $f'$ belongs to $\cW$ as well;
\item $\cW$ has the two-out-of-three property.
\end{itemize}
\end{definition}
 
\begin{lemma}[cf.~\citelist{\cite{uuye:hom}\cite{bel:kkstable}*{Proposition 2.10}}]\label{lem:cat-fib-objs}
The category $\CstAlg_s^G$ is a category of fibrant objects together with:
\begin{itemize}
\item the class of $\KK^G$-equivalences as weak equivalences; and
\item the class of semisplit surjections as fibrations.
\end{itemize}
\end{lemma}

\begin{proof}
The terminal object is given by the trivial algebra $0$.
Pullbacks exist in $\CstAlg^G$ in general, which are given by the usual pullback of vector spaces, that is, $V$ in~\eqref{eq:pullback} can be given by
\[
X \times_Z Y = \{ (x, y) \in X \times Y \mid f(x) = g(y) \}.
\]
This construction is compatible with surjective maps $f$ with completely positive splitting, since
\[
Y \to X \times_Z Y, \quad
y \mapsto (s(g(y)), y)
\]
would be a completely positive splitting for the map $f'(x, y) = y$ if $s$ is a one for $f$.

The only nontrivial remaining point is the compatibility for morphisms in $\cF \cap \cW$ with respect to the pullback.
We note that an equivariant $*$-homomorphism $f \colon X \to Y$ is in this intersection class if and only if it is in the class $\cF$ and its kernel $I$ is isomorphic to $0$ in the category $\KK^G$.
Indeed, by knowing that $f$ is in $\cF$, we have a $6$-term exact sequence
\[
\cdots \to \KK^G_i(Z, I) \to \KK^G_i(Z, X) \to \KK^G_i(Z, Y) \to \KK^G_{i + 1}(Z, I) \to \cdots
\]
for any $Z \in \CstAlg_s^G$, see~\cite{tu:hyp}*{Proposition 7.2}.
In particular, $f$ belongs to $\cF \cap \cW$ if and only if $I$ is isomorphic to $0$ in $\KK^G$.

Now, coming back to the pullback diagram~\eqref{eq:pullback}, by the above description we see that $\ker f'$ agrees with $\ker f$.
In particular, $f'$ is in $\cF \cap \cW$ if $f$ is in this class.
\end{proof}

\begin{remark}\label{rem:suspension-as-pullback}
Given a $G$-C$^*$-algebra $A$, the suspension $S A = C_0((0, 1)) \otimes A$ can be fit into the following pullback diagram in $\CstAlg_s^G$:
\[
\begin{tikzcd}
S A \arrow[r] \arrow[d] & C_0([1/2, 1)) \otimes A \arrow[d, "\mathrm{ev}_{1/2}"] \\
C_0((0, 1/2]) \otimes A \arrow[r, "\mathrm{ev}_{1/2}"] & A.
\end{tikzcd}
\]
The map $\mathrm{ev}_{1/2}\colon C_0((0, 1/2]) \otimes A \to A$ belongs to the class $\cF$, with the completely positive splitting map $a \mapsto t \otimes a$.
Of course, the same holds for the evaluation map from $C_0([1/2, 1)) \otimes A$.
\end{remark}

\begin{lemma}\label{lem:cat-fib-obj-extra-factoriz-prop}
Any morphism $f \colon A \to B$ in $\CstAlg_s^G$ can be factorized as $f = g \circ h$, with morphisms $h \in \cW$ and $g \in \cF$.
\end{lemma}

\begin{proof}
Let us take the mapping cylinder
\[
C = \{ (a, k) \in A \oplus C([0, 1], B) \mid f(a) = k(0) \}.
\]
Then the canonical map $h \colon A \to C$ given by $a \mapsto (a, k_{f(a)})$, where $k_b$ is the constant function with value $b$, is a homotopy inverse to the first factor projection $C \to A$, hence is a $\KK^G$-equivalence.
Moreover, the evaluation map $g \colon C \to B$ given by $(a, k) \mapsto k(1)$ is in the class $\cF$, with completely positive splitting given by $b \mapsto (0, t b)$.
\end{proof}

\begin{lemma}\label{lem:bott-periodicity}
In the $\infty$-category $\icKK^G$, we have a natural isomorphism $S^2 A \cong A$.
\end{lemma}

\begin{proof}
Let us consider the $C([0, 1])$-algebra $B = C_0([0, 1] \times \dR; A)$.
We consider the action $\alpha$ of $\dR$ on $B$ defined by
\[
\alpha_t(f(x, y)) = f(x, y + x t),
\]
and the associated crossed product $C = B \rtimes \dR$.
Then $C$ is still a $C([0, 1])$-algebra, and its fiber $C_0$ at $x = 0$ can be identified with $C_0(\dR^2) \otimes A \cong S^2 A$, while the fiber $C_x$ at any $x > 0$ can be identified with $C_0(\dR) \rtimes \dR \otimes A \cong \cK \otimes A$, with $\cK = \cK(\ell^2(\dN))$.

By~\cite{MR2511635}*{Proposition 2.2}, the evaluation maps $C \to C_x$ are all $\KK^G$-equivalences, hence isomorphisms in $\icKK^G$.
\end{proof}

\begin{proof}[Proof of Theorem~\ref{thm:kkg-infty-stable-tri}]
Lemmas~\ref{lem:cat-fib-objs} and~\ref{lem:cat-fib-obj-extra-factoriz-prop} imply that $(\CstAlg_s^G, \cW, \cF)$ is an $\infty$-category with weak equivalences and fibrations in the sense of~\cite{cis:ha}*{Definition 7.4.12}.
By~\cite{cis:ha}*{Proposition 7.5.6}, the canonical $\infty$-functor $L_\cW \colon \CstAlg_s^G \to \icKK^G$ is \emph{left-exact}, meaning that fibrations go to fibrations and pullbacks of the form~\eqref{eq:pullback} with $f \in \cF$ are mapped to a pullback diagram in the $\infty$-category $\icKK^G$.
Moreover, $\icKK^G$ has finite limits again by~\cite{cis:ha}*{Proposition 7.5.6}.

By~\cite{lur:ha}*{Corollary 1.4.2.27}, $\icKK^G$ will be a stable $\infty$-category if the loop functor $\Omega$, characterized by the pullback diagram
\[
\begin{tikzcd}
  \Omega A \arrow[r] \arrow[d] & 0 \arrow[d] \\
  0 \arrow[r] & A
\end{tikzcd}
\]
in $\icKK^G$, is an autoequivalence.
By the left exactness of $L_\cW$, we have a pullback diagram of the form
\[
\begin{tikzcd}
S A \arrow[r] \arrow[d] & 0 \arrow[d] \\
0 \arrow[r] & A,
\end{tikzcd}
\]
in $\icKK^G$, see Remark~\ref{rem:suspension-as-pullback}, hence $S A$ is a model of $\Omega A$.
By Lemma~\ref{lem:bott-periodicity}, $\Omega$ is indeed an autoequivalence.

Now, we know that $\icKK^G$ is stable, and moreover if a morphism $f \colon A \to B$ of $\CstAlg_s^G$ is in the class $\cF$, then the pullback diagram
\[
\begin{tikzcd}
  I \arrow[r] \arrow[d] & A \arrow[d, "f"] \\
  0 \arrow[r] & B
\end{tikzcd}
\]
with $I = \ker f$ in $\CstAlg_s^G$ goes to a pullback diagram in $\icKK^G$.
This means that, in the homotopy category $\ho(\icKK^G)$, we have a long exact sequence of the form
\[
\cdots \ho(\icKK^G)(Z, S B) \to \ho(\icKK^G)(Z, I) \to \ho(\icKK^G)(Z, A) \to \ho(\icKK^G)(Z, B) \to \cdots
\]
for any object $Z$.
Thus, the canonical functor $\CstAlg_s^G \to \ho(\icKK^G)$ is split (semi-)exact.
It is moreover homotopy invariant and stable, since these conditions can be captured by $\KK^G$-equivalence homomorphisms.
By the universality of $\KK^G$, we obtain a functor $\KK^G \to \ho(\icKK^G)$, which is the desired inverse of $\ho_0$.

It remains to compare the triangulated category structures.
A morphism in $\KK^G$ can be represented by an equivariant $*$-homomorphism $f \colon A \to B$ up to replacing objects by $\KK^G$-equivalences, see~\cite{laf:kkban}*{(proof of) Théorème A.2.2}.
Lemma~\ref{lem:cat-fib-obj-extra-factoriz-prop} implies that $f$ can be taken from the class $\cF$.
Then the kernel $I = \ker f$ is $\KK^G$-equivalent to the mapping cone of $f$.
This implies that the image of exact triangles of $\ho(\KK^G)$, which come from fibration sequences, contain the exact triangles of $\KK^G$, which are given by mapping cones.
In a triangulated category, any two exact triangles containing a given morphism must be isomorphic, hence we obtain the desired equivalence as triangulated categories.
\end{proof}

\subsection{Geometric realization}\label{subsec:georeal}

\begin{proposition}[\cite{bel:kkstable}*{Lemma 2.19}]
The stable $\infty$-category $\icKK^G$ admits countable colimits.
\end{proposition}

\begin{proof}
By stability, it is enough to show that it admits countable coproducts.
These can be detected in the homotopy category, so the claim follows from the standard fact that $\KK^G$ admits countable coproducts, given by the countable direct sum.
\end{proof}

\begin{corollary}
Let $A_\bullet = (A_n)_{n \in \dN}$ be a simplicial system of objects in $\icKK^G$, given by an $\infty$-functor from (the nerve of) $\Delta^\opo$ to $\icKK^G$.
Then the colimit $\colim_{\Delta^\opo} A_\bullet$ exists in $\icKK^G$.
\end{corollary}

\begin{definition}
We call $\colim_{\Delta^\opo} A_\bullet$ the \emph{geometric realization} of $A_\bullet$.
\end{definition}

An important example of countable colimit is the sequential colimit $A = \varinjlim A_n$ associated to an inductive system $\{A_n\}_{n\in\dN}$ in $\icKK^G$.
For example, the geometric realization $\colim_{\Delta^\opo} A_\bullet$  above can be written as a sequential colimit $\varinjlim D_n$, where $D_n = \colim_{\Delta^\opo_{\le n}} A_\bullet$ is the colimit of $A_\bullet$ restricted to the subcategory of $\Delta^\opo$ spanned by the objects $[m] $ for $m \le n$, i.e., the $n$-skeleton of $A_\bullet$.

\begin{proposition}[cf.~\cite{nestmeyer:loc}*{Section 2.4}]\label{prop:indinfcolim}
Let $(A_n)_n$ be an inductive system of objects, given by an $\infty$-functor $\dN \to \icKK^G$.
Its colimit can be computed in $\KK^G$ as the mapping cone of the map $\bigoplus_n A_n\to\bigoplus_n A_n$ given by $\id-\sigma$, where $\sigma$ is the diagonal action of the connecting maps $A_n \to A_{n + 1}$.
\end{proposition}

\begin{proof}
The equivalence between these two models goes as follows: the mapping cone is the coequalizer of the pair $(\id,\sigma)$ between the coproducts $\bigoplus_n A_n$.
This colimit can be rewritten as a colimit over a single category $I$, given by the poset with pairs $(i,j), i\geq 0, j\in\{0,1\}$ as objects, and morphisms $(i,j)\to (k,l)$ whenever $j=0,l=1$ and $i\leq k \leq i+1$.
The canonical projection $I\to \dN$ is cofinal, which implies the claim by~\cite{lur:ht}*{Proposition 4.1.1.8}.
\end{proof}

It is clear from the proof that Proposition~\ref{prop:indinfcolim} holds for any stable $\infty$-category (in place of $\icKK^G$) and its homotopy category with its natural triangulated structure (in place of $\KK^G$).

\section{Transfer maps}\label{sec:transfer-maps}

\subsection{\texorpdfstring{$\infty$}{Infinity}-functor on the category of proper correspondences}\label{app:alt-proof}

In this section we provide a construction of the map $\PrCorr^G \to \icKK^G$ introduced in Proposition~\ref{prop:simplicially-enriched-model-of-atiyah-transfer}.

\subsubsection{Additional preliminaries}

We also consider \emph{simplicially enriched categories}, that is, systems $\cC$ given by objects $X$, $Y$, \dots and simplicial sets $\cC_\bullet(X, Y)$ of ``morphisms'' together with composition maps
\[
\cC_n(Y, Z) \times \cC_n(X, Y) \to \cC_n(X, Z)
\]
compatible with the face and degeneracy maps, satisfying associativity and existence of units $\id_X \in \cC_0(X, X)$. If $\cC$ is a usual category, we take the constant simplicial set $\cC_n(X, Y) = \cC(X, Y)$ to treat $\cC$ as a simplicially enriched category.

Simplicially enriched categories $\cC$ whose morphism spaces $\cC_\bullet(X, Y)$ are Kan complexes provide another model of $(\infty, 1)$-categories. More precisely, there is a model structure on the category of simplicially enriched categories~\cite{berg:model}, and this category is Quillen equivalent to the Joyal model category of simplicial sets; see~\cite{lur:ht}*{Section 2.2} for a detailed proof.

Given a category $\cC$ enriched in Kan complexes, the corresponding quasicategory is given by the \emph{homotopy coherent nerve} $N^\hc(\cC)$. Let us quickly review the construction (see for example~\cite{land:inftybook}*{Definition 1.2.63}).
We denote the simplicial path category by $\Path[n]$, which has integers $0 \le k \le n$ as its objects and the nerve of partially ordered set $\{ I = \{ i = i_0 < i_1 < \cdots < i_m = j \} \}$ (with ordering given by the opposite of inclusion order) as the space of morphisms from $i$ to $j$ (which is empty if $j < i$).
The composition $\Mor_{\Path[n]}(j, k) \times \Mor_{\Path[n]}(i, j) \to \Mor_{\Path[n]}(i, k)$ is induced by the union of corresponding sets $I = \{ i = i_0 < i_1 < \cdots < i_m = j \}$ and $J = \{ j = j_0 < j_1 \cdots < j_l = k\}$.
Then $N^\hc_n(\cC)$ is given as the set of functors of simplicially enriched categories $\Path[n] \to \cC$.
At the level of $0$-cells in morphism spaces, compatibility with composition implies that the value of functors on the sets $I = \{i < j \}$ determines the value on the general case $I = \{ i = i_0 < i_1 < \cdots < i_m = j \}$.

If we start from an ordinary category $\cC$ with a collection of morphisms $\cW$ containing all identity morphisms, we get a model of $\infty$-localization given by a simplicially enriched category, called the \emph{hammock localization} $L^H(\cC,\cW)$ (also called simplicial or Dwyer--Kan localization)~\cite{dwkan:function}defined as follows.

Given two objects $X$ and $Y$ in $\cC$, we look at the zigzags of the form
\[
\begin{tikzcd}[column sep=small]
  X \arrow[r] & Z_1 & Z_2 \arrow[l] \arrow[r] & \cdots & Z_k \arrow[l] \arrow[r] & Y,
\end{tikzcd}
\]
where the leftward arrows are in $\cW$.
Two such zigzags are identified by removing leftward arrows when they are identity morphisms, and then composition of consecutive rightward arrows.
We consider the category of such equivalence classes, where a morphism is represented by a commutative diagram of the form
\[
\begin{tikzcd}[row sep=tiny, column sep=small]
   & Z_1 \arrow[dd] & Z_2 \arrow[l] \arrow[r] \arrow[dd] & \cdots & Z_k \arrow[l] \arrow[rd] \arrow[dd] & \\
  X \arrow[ru] \arrow[rd] & & & & & Y, \\
   & Z'_1 & Z'_2 \arrow[l] \arrow[r] & \cdots & Z'_k \arrow[l] \arrow[ru] &
\end{tikzcd}
\]
where all vertical arrows are morphisms in $\cW$.
Then the simplicial set $\Mor_{L^H(\cC, \cW)}(X, Y)$ is defined as the nerve of this category.

To get an $(\infty, 1)$-category, we apply Kan's $\Ex$ functor~\cite{kan:ex} infinitely many times to these simplicial sets and consider the category enriched in Kan complexes given by $\Ex_\infty \Mor_{L^H(\cC, \cW)}(X, Y)$. 

The equivalence between these two approaches, that is the comparison between the quasicategory $N(\cC)[\cW^{-1}]$ and the simplicially enriched category $\Ex_\infty L^H(\cC,\cW)$, is given by~\cite{hin:dwkanrev}*{Proposition 1.2.1}.

\subsubsection{Construction of transfer (proof of Proposition~\ref{prop:simplicially-enriched-model-of-atiyah-transfer})}

To simplify the notation we drop $G$ in the following discussion.
All constructions directly carry over to the equivariant case given the following points:
when $A$ is a $G$-C$^*$-algebra and $E_A$ is a right $G$-$A$-Hilbert module, the C$^*$-algebra $\cK(E_A)$ is also a $G$-C$^*$-algebra.
Moreover, if $E_A \to E_A \oplus E'_A$ is a canonical embedding, the induced homomorphism $\cK(E_A) \to \cK(E_A \oplus E'_A)$ is a $G$-equivariant homomorphism.

Recall that the simplicially enriched category $\Ex(L^H(\CstAlg_s, \cW_{\KK}))$ is the result of applying Kan's $\Ex$ functor to the simplicial sets of morphisms for the Hammock localization $L^H(\CstAlg_s, \cW_{\KK})$.
This is an intermediate step in replacing the morphism spaces of $L^H(\CstAlg_s, \cW_\KK)$ by Kan complexes, so that the homotopy coherent nerve gives a quasicategory representing the $\infty$-category $\icKK$.

We thus want to show that there is a map of simplicial sets
\[
T_\icKK\colon \ND (\PrCorr) \to N^{\hc}(\Ex(L^H(\CstAlg_s, \cW_{\KK}))).
\]
Recall that the simplicial sets $L^H(\CstAlg_s, \cW_{\KK})(A, B)$ are just the nerve of some category $\cC_{A, B}$.
Then the elements of $\Ex(L^H(\CstAlg_s, \cW_{\KK}))_n(A, B)$ are diagrams in $\cC_{A,B}$ of shape given by barycentric division of the $n$-simplex.
For example, a $1$-cell is given by a zigzag of hammocks $H_0 \to H_{01} \leftarrow H_1$, while a $2$-cell is given by a commutative diagram of the form
\begin{equation}\label{eq:hammock-diag-2}
\begin{tikzcd}[column sep=tiny, row sep=small]
& & H_1 \arrow[ld] \arrow[rd] & & \\
& H_{01} \arrow[r] & H_{012} & H_{12} \arrow[l]\\
H_0 \arrow[ru] \arrow[rr] & & H_{02} \arrow[u] & & H_2 \arrow[ll] \arrow[lu]
\end{tikzcd}
\end{equation}
representing composition of such zigzags.

Now, suppose that we have an element in $\ND_n(\PrCorr)$.
Unpacking the definition, it is given by a sequence of C$^*$-algebras $(A_i)_{i = 0}^n$, proper nondegenerate correspondences ${}_{A_{i}} E^{i, j}_{A_j}$ for $0 \le i < j \le n$, and a family of isomorphisms $E^{i, j} \otimes_{A_j} E^{j, k} \to E^{i, k}$ for $i < j < k$ satisfying the standard associativity condition.
Let us write $A_0 = A$, $A_n = B$, and $E^i = E^{i-1, i}$.

To describe its image in the $n$-th homotopy coherent nerve of $\Ex(L^H(\CstAlg_s, \cW_{\KK}))$, we need to produce a simplicial functor $F \colon \Path[n] \to \Ex(L^H(\CstAlg_s, \cW_{\KK}))$.
At the level of objects, we set $F(i) = A_i$.
At the level of morphisms, for the $0$-cell of $\Mor_{\Path[n]}(i, j)$ represented by $I = \{i, j\}$, we take the zigzag
\[
F(I) = A_i \to \cK_{A_j}(A_j \oplus E^{i,j}) \leftarrow A_j,
\]
where $A_i$ acts trivially (by $0$) on the summand $A_j \subset A_j \oplus E^{i, j}$, while the algebra $A_j$ embeds into $\cK_{A_j}(A_j \oplus E^{i,j})$ as $\cK_{A_j}(A_j)$.
By compatibility with composition, for a general $0$-cell $I = \{i = i_0 < \cdots < i_m = j \}$ in $\Mor_{\Path[n]}(i, j)$ the value of $F(I)$ should be given by
\begin{equation}\label{eq:zigzag-for-I}
A_i \to \cK_{A_{i_1}}(A_{i_1} \oplus E^{i, i_1}) \leftarrow A_{i_1} \to \cK_{A_{i_2}}(A_{i_2} \oplus E^{i_1, i_2}) \leftarrow \cdots \to \cK_{A_j}(A_j \oplus E^{i_{m-1},j}) \leftarrow A_j.
\end{equation}

Let us present the value of $F$ on higher cells.
A $k$-cell in the simplicial set $\Mor_{\Path[n]}(i, j)$ is given by a decreasing sequence of subsets $I_0 \supset I_1 \supset \cdots \supset I_k$ such that $\{i, j\} \subset I_l \subset [i, j]$.
We need to give a hammock diagram of zigzag objects $H_{l_1 l_2 \cdots l_m}$ fitting into diagrams of the form~\eqref{eq:hammock-diag-2}, such that $H_l$ is the zigzag~\eqref{eq:zigzag-for-I} for $I_l$.

We construct zigzags $H_{l_1 l_2}$, and in general set $H_{l_1 l_2 \cdots l_m} = H_{l_1 l_m}$.
Let us fix $l_1 < l_2$.
The inclusion $I_{l_2} \subset I_{l_1}$ defines subsets $P_{1}$, \dots, $P_{m}$ of $I_{l_1}$, as the intersection of $I_{l_1}$ and the intervals $[b_{x}, b_{x + 1}]$, with endpoints $b_x,b_{x+1}$ from $I_{l_2} = \{ b_0 < b_1 < \cdots < b_a\}$.
Let us write $P_i = \{ a^i_0 < \cdots < a^i_p \}$, and set $H_{l_1 l_2}^{(i)}$ to be the zigzag
\begin{multline}\label{eq:zigzag-for-block-Pi}
A_{a^i_0} \to \cK(A_{a^i_1} \oplus E^{a^i_0, a^i_1}) \to \cK(A_{a^i_2} \oplus E^{a^i_1,a^i_2} \oplus E^{a^i_0,a^i_1} E^{a^i_1,a^i_2}) \to \cdots\\
\to \cK(A_{a^i_p} \oplus E^{a^i_{p-1}, a^i_p} \oplus \cdots \oplus E^{a^i_0, a^i_1} \cdots E^{a^i_{p-1}, a^i_p}) \leftarrow A_{a^i_p},
\end{multline}
where we omitted the relative tensor product symbol.
In the second step, we let $\cK(A_{a^i_1} \oplus E^{a^i_0, a^i_1})$ act on the Hilbert $A_{a^i_2}$-module
\[
E^{a^i_1,a^i_2} \oplus E^{a^i_0,a^i_1} E^{a^i_1,a^i_2} = (A_{a_1^i} \oplus E^{a_0^i, a_1^i}) E^{a^i_1,a^i_2}
\]
using the induced action via tensor product with $E^{a^i_1, a^i_2}$.
To be precise, we interpret the above as a zigzag by formally adding leftward arrows $\cK(A_1 \oplus E^1) \leftarrow \cK(A_1 \oplus E^1)$, $\cK(A_2 \oplus E^2 \oplus E^1 E^2) \leftarrow \cK(A_2 \oplus E^2 \oplus E^1 E^2)$ and so on, given by the identity morphisms.
Note also that the defining relation of hammock localization implies that the above is equal to
\begin{equation}\label{eq:zigzag-for-block-simplified}
A_{a_0^i} \to \cK(A_{a^i_p} \oplus E^{a^i_{p-1}, a^i_p} \oplus \cdots \oplus E^{a^i_0, a^i_1} \cdots E^{a^i_{p-1}, a^i_p}) \leftarrow A_{a^i_p},
\end{equation}
where $A_{a_0^i}$ acts on the summand $E^{a^i_0, a^i_1} \cdots E^{a^i_{p-1}, a^i_p}$.
The zigzag $H_{l_1 l_2}$ is the composition of the $H^{(i)}_{l_1 l_2}$.

Let us describe the map $H_{l_1} \to H_{l_1 l_2}$.
By compatibility with composition, it is enough to describe it for each block $P_i$.
We then need to specify a map of zigzags from one of the form~\eqref{eq:zigzag-for-I}, with indexes $a^i_q$, to~\eqref{eq:zigzag-for-block-Pi}.
For each $q$, we can naturally embed $\cK(A_{a^i_q} \oplus E^{a^i_{q-1}, a^i_q})$ and $A_{a^i_q}$ to the first corners of $\cK(A_{a^i_q} \oplus E^{a^i_{q-1}, a^i_q} \oplus \cdots \oplus E^{a^i_0, a^i_1} \cdots E^{a^i_{q-1}, a^i_q})$, which defines the desired map.

To define $H_{l_2} \to H_{l_1 l_2}$, we use the isomorphism $E^{a^i_0, a^i_p} \to E^{a^i_0, a^i_1} \cdots E^{a^i_{q-1}, a^i_q}$ from the Duskin nerve structure to map the zigzag $A_{a^i_0} \to \cK(A_{a^i_p} \oplus E^{a^i_0, a^i_p}) \leftarrow A_{a^i_p}$ into~\eqref{eq:zigzag-for-block-simplified}.

Next let us describe the maps $H_{l_1 \cdots \hat{l}_j \cdots l_m} \to H_{l_1 \cdots l_m}$.
There are three cases depending on the value of the omitted index $j$.
In case of $1 < j < m$, we have $H_{l_1 \cdots \hat{l}_j \cdots l_m} = H_{l_1 l_m} = H_{l_1 \cdots l_m}$, hence we can take the identity map.
For $j = 1$, we need to give a map $H_{l_2 l_m} \to H_{l_1 l_m}$, while for $j = m$, we need to give a map $H_{l_1 l_{m-1}} \to H_{l_1 l_m}$.

The former case, $j = 1$, corresponds to having subsets $Q_i \subset P_i$, such that $P_i = \{ a^i_0 < \cdots < a^i_p \}$ is the intersection of $I_{l_1}$ with the interval with endpoints from $I_{l_m}$, and $Q_i = \{a^i_0 = b_0 < b_1 < \cdots < b_q = a^i_p \}$ is the intersection of $P_i$ with $I_{l_2}$.
Then we need to give a map from the zigzag
\[
A_{b_0} \to \cK(A_{b_q} \oplus E^{b_{q-1}, b_q} \oplus \cdots \oplus E^{b_0, b_1} \cdots E^{b_{q-1}, b_q}) \leftarrow A_{b_q}
\]
to~\eqref{eq:zigzag-for-block-simplified}.
We can do this by using the isomorphisms $E^{b_x, b_{x+1}} \cdots E^{b_{q-1}, b_q} \to E^{a^i_y, a^i_{y+1}} \cdots E^{a^i_{p-1}, a^i_p}$ from the Duskin nerve.

The latter case, $j = m$, corresponds to having a division $P_i = R^i_0 \cup \cdots \cup R^i_{r_i}$ into smaller intervals whose endpoints are from $I_{m-1}$.
Again the construction will be a composition of the ones on each $P_i$, so let us omit the index $i$ and write $P_i = \{ a'_0 < \cdots < a'_p\}$, $R^i_{j} = \{b'_{j, 0} < \cdots < b'_{j, k_j}\}$, so that we have $b'_{j, k_j} = b'_{j+1,0}$.
The hammock $H_{l_1 l_{m-1}}$ is a composite of
\begin{multline}\label{eq:R-j-block-hammock}
A_{b'_{j,0}} \to \cK(A_{b'_{j,1}} \oplus E^{b'_{j,0} b'_{j,1}}) \to \cdots \to \cK(A_{b'_{j,k_j}} \oplus E^{b'_{j,k_j-1} b'_{j,k_j}} \oplus \cdots \oplus E^{b'_{j,0} b'_{j,1}} \cdots E^{b'_{j,k_j-1} b'_{j,k_j}}) \\ 
\leftarrow A_{b'_{j,k_j}}.
\end{multline}
We want to construct a map of hammock into~\eqref{eq:zigzag-for-block-Pi}.
The hammock~\eqref{eq:R-j-block-hammock} at $j = 0$, except for the last term $A_{b'_{0,k_0}}$, matches the first part of~\eqref{eq:zigzag-for-block-Pi}, hence we can use identity morphisms.

The term $A_{b'_{0, k_0}} = A_{b'_{1, 0}}$ maps to the corresponding term of~\eqref{eq:zigzag-for-block-Pi} (namely, $\cK(A_{b'_{0,k_0}} \oplus E^{b'_{0,k_0-1} b'_{0,k_0}} \oplus \cdots \oplus E^{b'_{0,0} b'_{0,1}} \cdots E^{b'_{0,k_0-1} b'_{0,k_0}})$ with the identity arrow pointing left) by the first corner embedding, to have a map of hammocks.
The next term of~\eqref{eq:R-j-block-hammock} is $\cK(A_{b'_{1,1}} \oplus E^{b'_{1,0} b'_{1,1}})$ can be embedded in the first two corners of
\[
\cK(A_{b'_{1,1}} \oplus E^{b'_{1,0} b'_{1,1}} \oplus E^{b'_{0,k_0-1} b'_{0,k_0}} E^{b'_{1,0} b'_{1,1}} \oplus \cdots \oplus E^{b'_{0,0} b'_{0,1}} \cdots E^{b'_{0,k_0-1} b'_{0,k_0}} E^{b'_{1,0} b'_{1,1}}),
\]
which appears next in~\eqref{eq:zigzag-for-block-Pi}.
Continuing in the same way, we obtain a map $H_{l_1 l_{m-1}} \to H_{l_1 l_{m}}$.

In general, given a subsequence $l'_1, \ldots, l'_{m'}$ of $l_1, \ldots, l_m$, we can iterate the above construction, dropping one index at a time, to get a map of zigzags $H_{l'_1 \cdots l'_{m'}} \to H_{l_1 \cdots l_m}$. A standard bookkeeping argument shows this map is indeed well-defined, as the order in which indices are dropped does not affect the construction, which completes the construction.

\subsection{The functor \texorpdfstring{$\Map_c(-,\bE)$}{Map_c(-, E)} and sheaves}\label{subsec:mapcsheaves}

Let us give a topological model of $F^\bE$ classified by Theorem~\ref{thm:intro-thm-D}.
It will be convenient to work with CW-spectra.
Recall that a CW-spectrum is a topological spectrum $(\bE_n)_n$ such that each $\bE_n$ is a CW-complex and all structure maps are subcomplex inclusions. Any topological prespectrum receives a stable weak homotopy equivalence from a CW-spectrum~\cite{ekmm:modern}*{Theorem 1.5}. 

\begin{proposition}
Let $\bE_\bullet$ be a CW-spectrum.
Let $X$ be a totally disconnected compact space, and let $X = \varprojlim_n K_n$ be its presentation as a projective limit of finite spaces.
Then the natural map
\begin{equation}\label{eq:profin-dom-comparison}
\hocolim_n \Map(K_n, \bE_\bullet) \to \Map(X, \bE_\bullet)
\end{equation}
is a weak equivalence.
\end{proposition}

\begin{proof}
Let us fix $f \in \Map(X, \bE_n)$.
By the compactness of $X$, the image of $f$ is contained in a union of finite number of cells, $A = \bigcup_{i = 1}^k e_{\alpha_i}$.
Take a covering of $A$ by contractible relatively open sets $(A_j)_{j = 1}^N$, and take a partition of $X$ by clopen sets $X_j$ such that $f(X_j) \subset A_j$.
From this we see that $f$ is homotopic to a map with finite image.
This shows that the map~\eqref{eq:profin-dom-comparison} induces a surjective map of $\pi_0$-invariants.

A similar argument shows that any homotopy between $f_1, f_2 \in \Map(X, \bE_n)$ can be lifted to a homotopy in $\Map(K_m, \bE_n)$ for some big $m$, showing the injectivity of the map between $\pi_0$ induced by~\eqref{eq:profin-dom-comparison}.
The case of general $\pi_k$ can be reduced to the above by $\pi_k( \Map(X, \bE_\bullet))\cong \pi_0(\Map(X,\Omega^k\bE_\bullet))$ (note that $\Omega^k\bE_m$ has the homotopy type of a CW-complex).
\end{proof}

Let $Y$ be a second countable totally disconnected locally compact Hausdorff space.
Let us take an increasing sequence of open compact subspaces $(X_n)_{n = 1}^\infty$ in $Y$.
Given a topological spectrum $\bE_\bullet$, we define the \emph{compactly supported map spectrum} $\Map_c(Y, \bE_\bullet)$ to be the sequence of pointed spaces
\begin{equation}\label{eq:colimcomp}
\varinjlim_m \Map(X_m, \bE_n) \quad (n \in \dN),
\end{equation}
where the inductive limit can be considered strict (see Proposition~\ref{prop:trans-fun-spec}). Here, the connecting map $\Map(X_m, \bE_n) \to \Map(X_{m + 1}, \bE_n)$ is the ``extension by $0$'' map which sends $f$ to $\tilde{f}$, where $\tilde{f}$ sends the points of $X_{m + 1} \setminus X_m$ to the basepoint of $\bE_n$.
Note that this definition is independent of the choice of $(X_n)_n$.

\begin{proposition}\label{prop:trans-fun-spec}
The functor $Y \mapsto \Map_c(Y, \bE_\bullet)$ extends to a well-defined functor in the category $\Fun^{\times, cc}(\Span_{\pret}(\TDLC), \TopSpec)$. Hence $\Map_c(-, \bE_\bullet)$ provides a model for $F^{\bE}$.
\end{proposition}

\begin{proof}
This amounts to checking that, up to weak equivalence, we can replace the ordinary colimit with a homotopy colimit in~\eqref{eq:colimcomp}. Indeed, let $X_i$ be a sequential diagram of $T_1$-spaces with embedding bonding maps. Then for all $n\geq 0$, $\pi_n(\varinjlim_i X_i)\cong \varinjlim_i \pi_n(X_i)$ by~\cite{may:concise}*{Lemma on p.~67}. Since taking homotopy groups commutes with filtered homotopy colimits, the canonical map $\hocolim_i X_i \to \varinjlim_i X_i$ is a weak equivalence.
\end{proof}

It is not difficult to see that restricting $\Map(-,\bE)$ to the (opposite) category of compact totally disconnected spaces and continuous maps, viewed as a site with jointly surjective maps as covering, yields a sheaf in the sense of~\cite{lur:ht}*{Definition 6.2.2.6}. More precisely, if $X$ is a compact totally disconnected space, let us denote by $\underline{\bE}_X$ the sheafification of the constant presheaf on $X$ with value $\bE$. Then we have an equivalence of contravariant functors in $X$, 
\[
\Map(X,\bE)\cong \Gamma(X,\underline{\bE}_X)
\]
(the corresponding claim for sheaves of simplicial sets is proved in~\cite{lur:ht}*{Section 7.1}).
The unique extension to locally compact spaces is then equivalent to the compactly supported sections functor $X\mapsto \Gamma_c(X,\underline{\bE}_X)$, as defined in~\cite{lur:ha}*{Definition 5.5.5.9}.

Note that the descent property of sheaves, when restricted to the setting of totally disconnected spaces, is reduced to the product-preserving property we have imposed for the functors in Section~\ref{sec:span-tot-disc}. This suggests that, by imposing suitable descent conditions, one can lift the zero-dimensionality constraint and generalize the classification of functors from Section~\ref{sec:span-tot-disc} to the category $\Span_{\pret}(\LCH)$ of spans of locally compact Hausdorff spaces, in such a way that 
\[
\Gamma_c(-, \underline{\bE})\colon \Span_{\pret}(\LCH)\to \Spec
\]
is the ``model functor''. This is a generalization of a theorem attributed to Clausen in~\cite{knp-sheaves-on-manifolds}*{Theorem 3.6.9 and Remark 3.6.10}. A proof of this fact will appear in our upcoming work~\cite{py:upcoming}.

\raggedright

\end{document}